\documentclass[12pt]{amsart}
\usepackage{amssymb}
\usepackage{amsmath}
\usepackage{amsthm}
\newtheorem{theorem}{Theorem}[section]
\newtheorem{lemma}[theorem]{Lemma}

\newtheorem{definition}[theorem]{Definition}

\newtheorem{corollary}[theorem]{Corollary}
\newtheorem{proposition}[theorem]{Proposition}

\newtheorem{lem-def}[theorem]{Lemma-Definition}
\DeclareRobustCommand\longtwoheadrightarrow
{\relbar\joinrel\twoheadrightarrow}

\newcommand{\hooklongrightarrow}{\lhook\joinrel\longrightarrow}
\newcommand{\M}{\mathbb M}
\newcommand{\N}{\mathbb N}
\newcommand{\Z}{\mathbb Z}
\newcommand{\Q}{\mathbb Q}

\newcommand{\T}{\mathbb T}
\newcommand{\V}{\mathbb{V}}

\newcommand{\Vkp}{\mathbb{V}^{\operatorname{kp}}}
\def\P{\mathbb P}
\def\Pless{\mathbb P^{\op{dless}}}
\def\op{\operatorname}
\def\al{\alpha}
\def\as#1{\renewcommand\arraystretch{#1}}

\def\be{\beta}
\def\cc{{\mathcal C}}
\def\chr{\op{char}}

\def\e{\medskip}
\def\ep{\epsilon}
\def\erel{e_{\op{rel}}}
\def\g{\Gamma}
\def\ga{\gamma}
\def\gal{\op{Gal}}
\def\gen#1{\big\langle\, {#1} \,\big\rangle}
\def\gg{\mathcal{G}}
\def\ggm{\mathcal{G}_\mu}
\def\ggmp{\mathcal{G}_{\mu'}}
\def\gm{\g_\mu}
\def\gp{\mathfrak{p}}
\def\gr{\operatorname{gr}}

\def\hm{H_\mu}
\def\hmp{H_{\mu'}}
\def\imp{\,\Longrightarrow\,}
\def\im{\op{Im}}
\def\iso{\ \lower.3ex\hbox{\as{.08}$\begin{array}{c}\lra\\\mbox{\tiny $\sim\,$}\end{array}$}\ }
\def\k{\op{Ker}}
\def\kb{\overline{K}}
\def\km{k_\mu}

\def\kpm{\op{KP}(\mu)}
\def\kx{K[x]}
\def\la{\lambda}

\def\lc{\op{lc}}
\def\lg{l\raise.6ex\hbox to.2em{\hss.\hss}l}
\def\ll{\mathcal{L}}
\def\lra{\,\longrightarrow\,}

\def\m{{\mathfrak m}}
\def\minf{\mu_{-\infty}}

\def\mmu{\mid_\mu}
\def\mn{\op{Min}}
\def\mx{\op{Max}}
\def\rc{\op{rc}}
\def\nphm{N_{\mu,\phi}}
\def\npphm{N^{\mbox{\tiny pp}}_{\mu,\phi}}
\def\ok{\delta_0}
\def\om{\omega}
\def\oo{\mathcal{O}}
\def\ord{\op{ord}}
\def\phm{\phi_{\op{min}}}
\def\ppa{\mathcal{P}_{\alpha}}
\def\pset{\mathcal{P}}
\def\qg{\mathbb{Q}\g}
\def\qqg{\mathbb{Q}\times\mathbb{Q}\g}
\def\rep{\operatorname{Rep}}
\def\Res{\operatorname{Res}}
\def\rr{\mathcal{R}}
\def\rrm{\mathcal{R}_\mu}
\def\sep{_{\mbox{\tiny sep}}}

\def\sii{\,\Longleftrightarrow\,}
\def\smu{\sim_\mu}
\def\t{\theta}
\def\tame{^{\mbox{\tiny tame}}}
\def\T{\mathbb{T}}
\def\Tstr{\T^{\op{str}}}
\def\ty{\mathbf{t}}     

\title[Defectless polynomials and inductive valuations]{Defectless polynomials over henselian fields and inductive valuations}

\makeatletter
\@namedef{subjclassname@2010}{%
  \textup{2010} Mathematics Subject Classification}

\author[Moraes de Oliveira]{Nath\'alia Moraes de Oliveira}
\address{Departament de Matem\`{a}tiques,
         Universitat Aut\`{o}noma de Barcelona,
         Edifici C, E-08193 Bellaterra, Barcelona, Catalonia, Spain}
\email{noliveira@mat.uab.cat,\quad nart@mat.uab.cat}

\author[Nart]{Enric Nart}
\thanks{Partially supported by grants MTM2016-75980-P from MEC, and 204224/2014-4 from CNPq}
\date{}
\keywords{defectless polynomial, graded algebra of a valuation, henselian field, key polynomial, MacLane chain, Newton polygon, residual polynomial operator, inductive valuation}

\begin{document}

\begin{abstract}
Let $(K,v)$ be a henselian valued field. Let $\Pless\subset\kx$ be the set of monic, irreducible polynomials which are defectless and have degree greater than one. For a certain equivalence relation $\,\approx\,$  on $\,\Pless$, we establish a canonical bijection $\M\to\Pless/\!\!\approx$, where $\M$ is a discrete \emph{MacLane space}, constructed in terms of inductive valuations on $\kx$ extending $v$.
\end{abstract}

\maketitle%%

\section*{Introduction}
Let $(K,v)$ be a valued field, with value group $\g=v(K^*)$. 
In a pioneering work, S. MacLane studied the
extensions of the valuation $v$ to the polynomial ring $\kx$ in one 
indeterminate, in the case $v$ discrete of rank one \cite{mcla,mclb}. 

Starting with any extension $\mu_0$ on $\kx$, he considered 
\emph{inductive valuations} $\mu$, obtained afer a finite number of augmentation steps:
\begin{equation}\label{depthintro}
\mu_0\ \stackrel{\phi_1,\ga_1}\lra\  \mu_1\ \stackrel{\phi_2,\ga_2}\lra\ \cdots
\ \stackrel{\phi_{r-1},\ga_{r-1}}\lra\ \mu_{r-1} 
\ \stackrel{\phi_{r},\ga_{r}}\lra\ \mu_{r}=\mu,
\end{equation}
involving the choice of certain \emph{key polynomials} $\phi_i \in K[x]$ and elements $\ga_i\in\g\otimes\Q$.

MacLane proved that all extensions of $v$ to $\kx$ can be obtained as a certain limit of inductive valuations. These ideas lead to an efficient algorithm for polynomial factorization over $K_v[x]$, where $K_v$ is the completion of $K$ at $v$. 

This algorithm facilitates an efficient resolution of many arithmetic-geometric tasks in number fields and function fields of algebraic curves \cite{gen}.

M. Vaqui\'e generalized MacLane's theory to arbitrary valued fields \cite{Vaq}. The graded algebra $\gg_{\mu}$  attached to a valuation $\mu$ on $\kx$, and some ideal theory considerations in the degree-zero subring $\Delta_{\mu}\subset\gg_{\mu}$, are crucial for the development of the theory.

Let $(K^h,v^h)$ be a henselization of $(K,v)$.
In analogy with the  discrete, rank-one case, MacLane-Vaqui\'e's theory should lead to the  design of efficient 
polynomial factorization algorithms over $K^h[x]$, which in turn could contribute to 
the computational resolution of arithmetic-geometric tasks in algebraic varieties of higher dimension.

For instance, F.J. Herrera, W. Mahboub, M.A. Olalla and M.Spivakovsky use simi\-lar key polynomials as a tool to attack the local uniformization theorem for quasi-excellent noetherian schemes in positive and mixed characteristic \cite{hos,hmos}.  

In our approach, the underlying basic idea is to approximate a prime (monic, irreducible) polynomial $F\in K^h[x]$ by adequate key polynomials of (not necessarily inductive) valuations $\mu$ on $\kx$ such that $\mu(F)$ is sufficiently large. 

If we restrict our attention to inductive valuations, their key polynomials  can approximate only \emph{defectless} prime polynomials of henselian valued fields, as shown by Vaqui\'e in \cite{Vaq2}.  A prime polynomial $F\in K^h[x]$ is defectless if $\deg(F)=e(F)f(F)$, where $e(F)$, $f(F)$ are the ramification index and residual degree, respectively, of the extension of $K^h$ obtained by adjoining a root of $F$.

In this paper, we study constructive methods to produce approximations to defectless prime polynomials, by key polynomials for inductive valuations. To this end, we generalize to arbitrary henselian fields the results of \cite{ResidualIdeals}, where only the discrete rank-one case was considered.

Let $\Pless\subset K^h[x]$ be the subset of prime defectless polynomials of degree greater than one. We define an \emph{Okutsu equivalence relation} $\,\approx\,$  on $\Pless$ as follows: 
$$
F\approx G \ \sii\ \deg(F)=\deg(G),\quad v^h(\Res(F,G))>\deg(F)^2w(F),
$$
for a certain bound $w(F)\in\g\otimes\Q$ defined in section \ref{secAppr}. 

If $F,G\in\Pless$ are separable and Okutsu equivalent, they determine two extensions of $K^h$ with isomorphic maximal tame subextensions (Theorem \ref{tame}).

Our main result parameterizes the quotient set $\,\Pless/\!\approx\;$ by a discrete \emph{MacLane space} $\M$, constructed in terms of inductive valuations (Theorem \ref{MLspace}). 

The space $\M$ consists of pairs $(\mu,\ll)$, where $\mu$ is an inductive valuation and $\ll$ is a (strong) maximal ideal of $\Delta_\mu$. 
The maximal spectrum of $\Delta_\mu$ is in canonical bijection with the set  of $\mu$-equivalence classes of key polynomials for $\mu$ \cite{KeyPol}. Hence, we may assign to each pair $(\mu,\ll)\in\M$ the $\mu$-equivalence class of key polynomials determined by $\ll$, which coincides with an Okutsu equivalence class in $\Pless$.

All prime defectless polynomials in the class attached to  $(\mu,\ll)$ share all discrete arithmetic invariants deduced from (optimal) MacLane chains of $\mu$, as in (\ref{depthintro}).

These results, combined with the computational techniques of \cite{CompRP}, yield an algorithm for polynomial factorization in $K^h[x]$ of separable defectless polynomials in $\kx$; that is, separable polynomials all whose prime factors in $K^h[x]$ are defectless. 

For each prime factor $F\in K^h[x]$, this algorithm computes an Okutsu equivalent approximation to $F$ in $\kx$, together with all arithmetic invariants of $F$ that can be read in an optimal MacLane chain of the inductive valuation $\mu$ corresponding to $F$ through the above mentioned bijection.

\section{Valuations on polynomial rings and key polynomials}\label{secVals}
\pagestyle{headings}

Throughout this paper, we fix a valued field $(K,v)$. Let
$$\m\subset\oo\subset K, \qquad  k=\oo/\m$$ be the maximal ideal, valuation ring and residue class field of the valuation $v$.

Let $\g=v(K^*)$ be the value group, and denote $\qg=\g\otimes\Q$.

In this section, we review some basic facts on valuations on $K[x]$, mainly extracted from \cite{CompRP}, \cite{KeyPol} and \cite{Vaq}.
%For any $a \in \oo$, we denote by $\overline{a} \in  k$ its class modulo $\m$. 

%Given a polynomial $g \in \oo[x]$, we denote by $\overline{g} \in  k[x]$ the polynomial which is obtained by taking classes modulo $\m$ of all coefficients of $g$.\e

\subsection{Graded algebra of a valuation on $\kx$ and key polynomials}\label{subsecGr}
Let $\mu$ be a valuation on $\kx$ extending $v$. It extends in an obvious way to a valuation on $K(x)$.

Let $\g_\mu=\mu\left(K(x)^*\right)$ be the value group, and denote the maximal ideal, valuation ring and residue class field, by
$$\m_\mu\subset \oo_\mu\subset K(x),\qquad k_{\mu}=\oo_\mu/\m_\mu.$$

For any $\alpha\in\g_\mu$, consider the abelian groups:
$$
\ppa=\{g\in \kx\mid \mu(g)\ge \alpha\}\supset
\ppa^+=\{g\in \kx\mid \mu(g)> \alpha\}.
$$    
%Clearly, $\pset_0$ is a subring of $\kx$, and $\ppa$, $\ppa^+$
%are $\pset_0$-submodules of $\kx$ for all $\alpha$.

%\begin{definition}\label{gr}
The \emph{graded algebra of $\mu$ over $\kx$} is the integral domain:
$$
\ggm=\gr_{\mu}(\kx)=\bigoplus\nolimits_{\alpha\in\g_\mu}\ppa/\ppa^+.
$$
%\end{definition}

There is a natural map $\hm\colon \kx\to \ggm$, given by $\hm(0)=0$ and
$$\hm(g)= g+\pset_{\mu(g)}^+\in\pset_{\mu(g)}/\pset_{\mu(g)}^+, \quad\mbox{ if }g\ne0.
$$
Note that $\hm(g)\ne0$ if $g\ne0$. For all $g,h\in \kx$ we have:
\begin{equation}\label{Hmu}
\as{1.3}
\begin{array}{l}
 \hm(gh)=\hm(g)\hm(h), \\
 \hm(g+h)=\hm(g)+\hm(h), \ \mbox{ if }\mu(g)=\mu(h)=\mu(g+h).
\end{array}
\end{equation}

%The next definitions translate properties of the action of  $\mu$ on $\kx$ into algebraic relationships in the graded algebra $\ggm$.

\begin{definition}\label{mu}Let $g,\,h\in \kx$.

We say that $g,h$ are \emph{$\mu$-equivalent}, and we write $g\smu h$, if $\hm(g)=\hm(h)$. 

We say that $g$ is \emph{$\mu$-divisible} by $h$, and we write $h\mmu g$, if $\hm(h)\mid \hm(g)$ in $\ggm$. 

We say that $g$ is $\mu$-irreducible if $\hm(g)\ggm$ is a non-zero prime ideal. 

We say that $g$ is $\mu$-minimal if $g\nmid_\mu f$ for all non-zero $f\in \kx$ with $\deg(f)<\deg(g)$.
\end{definition}

The property of $\mu$-minimality admits a relevant characterization.

\begin{lemma}\label{minimal0}
Let $\phi\in \kx$ be a non-constant polynomial. Let 
\begin{equation}\label{phiexp}
f=\sum\nolimits_{0\le s}a_s\phi^s, \qquad  a_s\in \kx,\quad \deg(a_s)<\deg(\phi) 
\end{equation}
be the canonical $\phi$-expansion of $f\in \kx$.
Then, $\phi$ is $\mu$-minimal if and only if
$$
\mu(f)=\mn\{\mu(a_s\phi^s)\mid 0\le s\},\quad \forall\,f\in \kx.
$$
\end{lemma}

\begin{definition} 
A MacLane-Vaqui\'e \emph{key polynomial} for $\mu$ is a monic polynomial in $\kx$ which is simultaneously  $\mu$-minimal and $\mu$-irreducible. 

The set of key polynomials for $\mu$ will be denoted $\kpm$.

A key polynomial is necessarily irreducible in $\kx$.
\end{definition}

\begin{lemma}\label{mid=sim}Let $\phi\in\kpm$, and let $f\in \kx$ be a monic polynomial such that $\phi\mmu f$ and $\deg(f)=\deg(\phi)$. Then, $\phi\smu f$ and $f$ is a key polynomial for $\mu$ too.\end{lemma}

The \emph{weight} of $\mu$ is the following upper bound $w(\mu)\in\qg$ for weighted $\mu$-values.

\begin{theorem}\label{bound}
Let $\phi\in\kpm$. For any monic non-constant $f\in \kx$ we have
$$
\mu(f)/\deg(f)\le w(\mu):=\mu(\phi)/\deg(\phi),
$$
and equality holds if and only if $f$ is $\mu$-minimal.
\end{theorem}

%\begin{definition}\label{RRmu}
Let $\Delta=\Delta_\mu=\pset_0/\pset_0^+\subset\ggm$ be the subring of homogeneous elements of degree zero. 
There are canonical injective ring homomorphisms: 
$$ k\hooklongrightarrow\Delta\hooklongrightarrow k_{\mu}.$$
%In particular, $\Delta$ and $\ggm$ are equipped with a canonical structure of $ k$-algebra. 
%\begin{lemma}\label{maxsubfield}
The algebraic closure of $k$ in $\Delta$ is a subfield $\kappa=\kappa(\mu)\subset\Delta$ such that $\kappa^*=\Delta^*$. 
%\end{lemma}

Let $I(\Delta)$ be the set of ideals in $\Delta$, and consider the \emph{residual ideal operator}:
$$
\rr=\rrm\colon \kx\lra I(\Delta),\qquad g\longmapsto \left(\hm(g)\ggm\right)\cap \Delta.
$$
%\end{definition}

This operator $\rr$ translates questions about the action of  $\mu$ on $\kx$ into ideal-theoretic consi\-derations in the $k$-algebra $\Delta$.

\begin{theorem}\label{Max}
If $\kpm\ne\emptyset$, the residual ideal operator induces a bijection: 
$$\rr\colon \kpm/\!\!\smu\,\lra\,\mx(\Delta).$$ 
\end{theorem}

\subsection*{Valuation attached to a key polynomial}
For any positive integer $m$ we denote:
$$
%\as{1.4}\begin{array}{l}K[x]_m=\left\{a\in K[x]\mid \deg(a)<m\right\},\\
\g_{\mu,m}=\left\{\mu(a)\in\gm\mid a\in K[x],\ 0\le \deg(a)<m\right\}\subset\gm.
%\end{array}
$$

\begin{proposition}\label{vphi}
Take $\phi\in\kpm$ and let $m=\deg(\phi)$. Consider the maximal ideal $\gp=\phi\kx$, the field $K_\phi=\kx/\gp$,
and the onto mapping:
$$
v_{\mu,\phi}\colon K_\phi^*\longtwoheadrightarrow \g_{\mu,m},\qquad v_{\mu,\phi}(f+\gp)=\mu(a_0),\quad \forall f\in\kx\setminus\gp,
$$
where $a_0\in\kx$ is the remainder of the division of $f$ by $\phi$. Then,
$v_{\mu,\phi}$ is a valuation on $K_\phi$ extending $v$, with value group $\g_{\mu,m}$. Moreover,
$\gm=\gen{\g_{\mu,m},\mu(\phi)}$.
\end{proposition}

Denote the maximal ideal, the valuation ring and the residue class field of $v_{\mu,\phi}$ by: 
$$\m_{\mu,\phi}\subset\oo_{\mu,\phi}\subset K_\phi,\qquad k_{\mu,\phi}=\oo_{\mu,\phi}/\m_{\mu,\phi}.$$ 

\begin{proposition}\label{maxideal}
If $\phi\in\kpm$, then 
$\rr(\phi)$ is the kernel of the onto homomorphism $$\Delta\longtwoheadrightarrow k_{\mu,\phi},\qquad g+\pset^+_0\ \longmapsto\ \left(g+\gp\right)+\m_{\mu,\phi}.$$
If $\phi$ has minimal degree in $\kpm$, the mapping $\kappa\hookrightarrow\Delta\twoheadrightarrow k_{\mu,\phi}$ is an isomorphism. 
\end{proposition}

Thus, if $\kpm\ne\emptyset$, we have $[\kappa\colon k]<\infty$. 
We define the \emph{residual degree} of $\mu$ as
$$
f(\mu):=[\kappa\colon k].
$$
%By Proposition \ref{maxideal}, $f(\mu)=f(v_{\mu,\phi}/v)$ for all $\phi\in\kpm$ of minimal degree.

\subsection{Commensurable extensions}\label{subsecComm}

The extension $\mu/v$ is \emph{commensurable} if $\g_\mu/\g$ is a torsion group. In this case, there is a canonical embedding $$\g_\mu\hooklongrightarrow \qg.$$

By composing $\mu$ with this embedding, we obtain a $\qg$-valued valuation on $\kx$. 
%$$\kx \longrightarrow \qg \cup \lbrace\infty\rbrace.$$
Conversely, any $\qg$-valued extension of $v$ to $\kx$ is commensurable over $v$.

Two commensurable extensions of $v$ are equivalent if and only if their corresponding $\qg$-valued valuations coincide.
Hence, we may identify the set of equivalence classes of comensurable valuations extending $v$ with the set
$$\V:=\V(K,v)=\lbrace \mu\colon \kx\lra \qg \cup \lbrace\infty\rbrace \mid \mu\mbox{ valuation, } \mu_{\vert K}=v\rbrace.$$

There is a natural partial ordering in the set of all mappings from $\kx$ to any fixed ordered group:
$$\mu\le \mu' \quad\mbox{ if }\quad\mu(f) \le \mu'(f), \quad \forall\,f\in \kx. $$
In particular, $(\V,\le)$ is a partially ordered set.

\begin{lemma}\label{residTr}
Let $\mu$ be a valuation on $\kx$ extending $v$. The following conditions are equivalent.
\begin{enumerate}
\item $\mu/v$ is commensurable and $\kpm\ne\emptyset$.
\item $\mu$ is residually transcendental; that is, $k_\mu/k$ is a transcendental extension.
\item $\mu/v$ is commensurable and $\mu$ is not maximal in $(\V,\le)$.
\end{enumerate}
\end{lemma}

%In fact, the equivalence (1)  $\Leftrightarrow$ (2) is well known \cite{PP}, \cite[Thms. 4.2,4.4]{KeyPol}.

%The equivalence (1) $\Leftrightarrow$ (3) follows from the augmentation of valuations (section \ref{subsecAugmentation}) and \cite[Lem. 1.15]{Vaq} (cf. Lemma \ref{minDegree}).  

Let us denote the set of equivalence classes of commensurable extensions of $v$ admitting key polynomials by
$$
\Vkp=\Vkp(K,v)\subset \V.
$$

For $\mu\in\Vkp$, let $\phm$ be a key polynomial for $\mu$ of minimal degree $m$ in $\kpm$. By Proposition \ref{vphi}, $\g_{\mu,m}= \g_{v_{\mu,\phm}}$ and $\gm/\g_{\mu,m}$ is a finite cyclic group generated by $\mu(\phm)$.
Thus, we may define the \emph{absolute and relative ramification indices} of $\mu$ as
$$
e(\mu):=\left(\gm\colon\g\right)<\infty,\qquad \erel(\mu):=\left(\gm\colon \g_{\mu,m}\right)<\infty.
$$

Let us describe the structure of the residue class field $k_\mu$.
For all $a\in\kx$ with $\deg(a)<m$ the homogeneous element $\hm(a)$ is a unit in $\ggm$ \cite[Prop. 3.5]{KeyPol}. 

\begin{theorem}\label{kstructure}\mbox{\null}
For $\mu\in\Vkp$, take $\phm\in\kpm$ of minimal degree $m$. Let  $e=\erel(\mu)$ and take $a\in\kx$ with $\deg(a)<m$ such that $\mu(a)=-e\mu(\phm)$. Denote
$$\ep=\hm(a)\in\ggm^*,\qquad \xi=\hm(\phm)^e\ep\in\Delta.$$

Then, $\xi$ is transcendental over $\kappa$, and
$$\Delta=\kappa[\xi],\qquad \op{Frac}(\Delta)=\kappa(\xi)\simeq \km,$$
the last isomorphism being induced by the canonical embedding $\Delta\hookrightarrow \km$.
\end{theorem}

\subsection{Newton polygons}\label{subsecNewton}
Let $\mu$ be a valuation in the space $\Vkp$.

The choice of a key polynomial $\phi$ for $\mu$ determines a \emph{Newton polygon operator}
$$
\nphm\colon\, \kx\lra \pset\left({\qqg}\right),
$$
where $\pset\left({\qqg}\right)$ is the set of subsets of the rational space $\qqg$. 

The Newton polygon of the zero polynomial is the empty set. 

If $f\in \kx$ is nonzero and has a canonical $\phi$-expansion as in (\ref{phiexp}),
then $N:=\nphm(f)$ is defined to be the lower convex hull of the finite cloud of points 
\begin{equation}\label{cloud}
\cc=\left\{Q_s\mid s\ge0\right\},\qquad Q_s=\left(s,\mu\left(a_s\right)\right).
\end{equation}

Thus, $N$ is either a single point or a chain of segments, ordered from left to right by increasing slopes, called the \emph{sides} of the polygon.

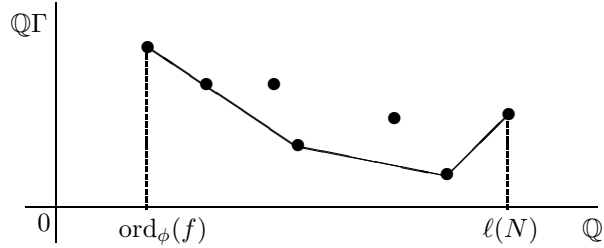
\begin{figure}%[h]
\caption{Newton polygon $N=\nphm(f)$ of $f\in \kx$}\label{figNmodel}
\begin{center}
\setlength{\unitlength}{4mm}
\begin{picture}(20,7.5)
\put(14.8,3.4){$\bullet$}\put(11,3.3){$\bullet$}\put(12.75,1.4){$\bullet$}\put(7,4.4){$\bullet$}
\put(7.8,2.4){$\bullet$}\put(4.75,4.4){$\bullet$}\put(2.8,5.65){$\bullet$}
\put(-1,0.6){\line(1,0){19}}\put(0,-0.6){\line(0,1){8}}
\put(8,2.6){\line(-3,2){5}}\put(8,2.63){\line(-3,2){5}}
\put(8,2.6){\line(5,-1){5}}\put(8,2.62){\line(5,-1){5}}
\put(13,1.6){\line(1,1){2}}\put(13,1.62){\line(1,1){2}}
%\put(15,3.6){\line(1,3){1}}\put(15,3.62){\line(1,3){1}}
\multiput(3,.5)(0,.25){22}{\vrule height2pt}
%\multiput(8,.9)(0,.25){9}{\vrule height2pt}
\multiput(15,.5)(0,.25){12}{\vrule height2pt}
%\put(7.2,0){\begin{footnotesize}$s_{\mu,\phi}(g)$\end{footnotesize}}
\put(2.1,-0.4){\begin{footnotesize}$\ord_{\phi}(f)$\end{footnotesize}}
\put(14.2,-0.4){\begin{footnotesize}$\ell(N)$\end{footnotesize}}
\put(17.5,-0.4){\begin{footnotesize}$\Q$\end{footnotesize}}
\put(-1.5,6.4){\begin{footnotesize}$\qg$\end{footnotesize}}
%\multiput(-.1,3)(.25,0){55}{\hbox to 2pt{\hrulefill }}
%\put(6,7){\begin{footnotesize}\end{footnotesize}}
%\put(-1.9,2.85){\begin{footnotesize}$\mu(g)$\end{footnotesize}}
\put(-.6,-.2){\begin{footnotesize}$0$\end{footnotesize}}
\end{picture}
\end{center}
\end{figure}

The abscissa of the left endpoint of $N$ is $\ord_{\phi}(f)$ in $\kx$.  

The abscissa of the right endpoint of $N$ is called the \emph{length} of $N$, and is denoted:
$$
\ell(N)=\left\lfloor \deg(f)/\deg(\phi)\right\rfloor.
$$
The left and right endpoints of $N$, together with the points joining two different sides are called \emph{vertices} of $N$.

\begin{definition}\label{sla}
Let $f\in\kx$, $N=\nphm(f)$ and  $\la\in\qg$.

The \mbox{\emph{$\la$-component}} $S_\la(N)\subset N$ is the intersection of $N$ with the line of slope $-\la$ which first touches $N$ from below. In other  words,
$$S_\la(N)= \{(s,u)\in N\,\mid\, u+s\la\mbox{ is minimal}\,\}.$$

The abscissas of the endpoints of $S_\la(N)$ are denoted \ $s_{\la}(f)\le s'_{\la}(f)$.

We say that $N=\nphm(f)$ is \emph{one-sided} of slope $-\la$ if 
$$
N=S_\la(N),\qquad s_{\la}(f)=0,\qquad s'_{\la}(f)>0. 
$$
\end{definition}

If $N$ has a side $S$ of slope $-\la$, then $S_\la(N)=S$. Otherwise, $S_\la(N)$ is a vertex of $N$.  Figure \ref{figComponent0} illustrates both possibilities.

%\begin{definition}\label{onesided}\end{definition}

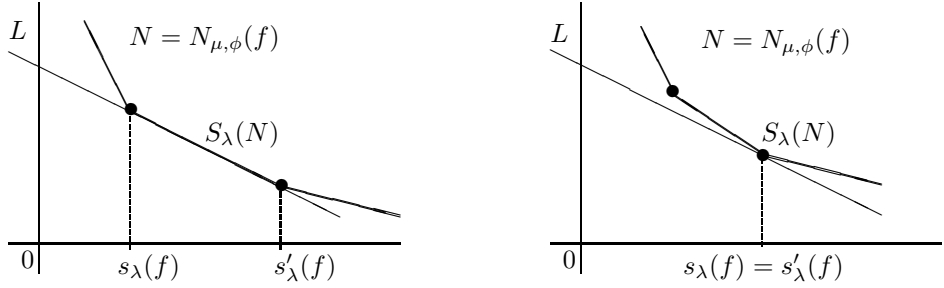
\begin{figure}%[h]
\caption{$\la$-component of $\nphm(f)$. The line $L$ has slope $-\la$.}\label{figComponent0}
\begin{center}
\setlength{\unitlength}{4mm}
\begin{picture}(30,8.8)
\put(2.8,4.8){$\bullet$}\put(7.8,2.3){$\bullet$}
\put(-1,0.6){\line(1,0){13}}\put(0,-0.4){\line(0,1){9}}
\put(-1,7){\line(2,-1){11}}
\put(3,5){\line(-1,2){1.5}}\put(3,5.04){\line(-1,2){1.5}}
\put(3,5){\line(2,-1){5}}\put(3,5.04){\line(2,-1){5}}
\put(8,2.5){\line(4,-1){4}}\put(8,2.54){\line(4,-1){4}}
\multiput(3,.4)(0,.25){18}{\vrule height2pt}
\multiput(8,.4)(0,.25){8}{\vrule height2pt}
\put(7.8,-.4){\begin{footnotesize}$s'_{\la}(f)$\end{footnotesize}}
\put(2.6,-.4){\begin{footnotesize}$s_{\la}(f)$\end{footnotesize}}
\put(-1,7.4){\begin{footnotesize}$L$\end{footnotesize}}
\put(-.6,-.2){\begin{footnotesize}$0$\end{footnotesize}}
\put(3,7.2){\begin{footnotesize}$N=\nphm(f)$\end{footnotesize}}
\put(5.5,4){\begin{footnotesize}$S_\la(N)$\end{footnotesize}}
\put(20.8,5.4){$\bullet$}\put(23.8,3.3){$\bullet$}
\put(17,0.6){\line(1,0){13}}\put(18,-0.4){\line(0,1){9}}
\put(17,7.05){\line(2,-1){11}}
\put(24,3.6){\line(-3,2){3}}\put(24,3.64){\line(-3,2){3}}
\put(21,5.8){\line(-1,2){1}}\put(21,5.84){\line(-1,2){1}}
\put(24,3.55){\line(4,-1){4}}\put(24,3.59){\line(4,-1){4}}
\multiput(24,.5)(0,.25){12}{\vrule height2pt}
\put(21.4,-.4){\begin{footnotesize}$s_{\la}(f)=s'_{\la}(f)$\end{footnotesize}}
\put(17,7.3){\begin{footnotesize}$L$\end{footnotesize}}
\put(17.4,-.2){\begin{footnotesize}$0$\end{footnotesize}}
\put(22,7.1){\begin{footnotesize}$N=\nphm(f)$\end{footnotesize}}
\put(24,4){\begin{footnotesize}$S_\la(N)$\end{footnotesize}}
\end{picture}
\end{center}
\end{figure}

%\subsection*{Principal Newton polygons}
Since $\phi$ is $\mu$-minimal, Lemma \ref{minimal0} shows that $$\mu(f)=\mn\{\mu\left(a_s\phi^s\right)\mid s\ge0\}=\mn\{\mu\left(a_s\right)+s\mu(\phi)\mid s\ge0\}.$$ 
Hence, $\mu(f)$ is the ordinate of the point where the vertical axis meets the line of slope $-\mu(\phi)$ containing the $\mu(\phi)$-component of $\nphm(f)$. (see Figure \ref{figAlpha})

%Let $S$ be a side of $\nphm(f)$ with slope $\ga>\mu(\phi)$. Then, the augmented valuation $[\mu;\phi,\ga]$ contains relevant arihmetic information about $f$, with respect to $\phi$ (see Theorem \ref{main}). This motivates the next definition.

\begin{definition}
The \emph{principal Newton polygon} $\npphm(f)$ is the polygon formed by the sides of $\nphm(f)$ of slope less than $-\mu(\phi)$. 

If $\nphm(f)$ has no sides of slope less than $-\mu(\phi)$, then $\nphm^{\mbox{\tiny pp}}(f)$ is defined to be the left endpoint of $\nphm(f)$. 
\end{definition}

\begin{lemma}\label{length}
The integer $\ell\left(\nphm^{\mbox{\tiny pp}}(f)\right)=s_{\mu(\phi)}(f)$  is the order with which the prime element $\hm(\phi)$ divides $\hm(f)$ in the graded algebra $\ggm$.
\end{lemma}

There is a natural addition law for Newton polygons. 
Consider two  polygons $N$, $N'$ with sides $S_1,\cdots,S_k$; $S'_1,\dots,S'_{k'}$, respectively. 

The left endpoint of the sum $N+N'$ is  the vector sum in $\qqg$ of the left endpoints of $N$ and $N'$, whereas the sides of $N+N'$ are obtained by joining all sides in the multiset $\left\{S_1,\dots,S_k,S'_1,\dots,S'_{k'}\right\}$, ordered by increasing slopes.

\begin{theorem}\label{product}
For any non-zero $g,h\in \kx$ we have $$\npphm(gh)=\npphm(g)+\npphm(h).$$
\end{theorem}

\subsection{Residual polynomial operator}\label{subsecRi} 
For $\mu\in\Vkp$, take $\phm\in\kpm$ of minimal degree $m$. Denote  
$e=\erel(\mu)$ and $\ga=\mu(\phm)$. 

Let $\ep\in\ggm^*$ be a unit of degree $-e\ga$. By Theorem \ref{kstructure}, the  element $\xi=\hm(\phm)^e\ep$ is transcendental over $\kappa$ and satisfies $\Delta=\kappa[\xi]$.

The choice of the pair $\phm,\ep$ determines a residual polynomial operator;
$$
R:=R_{\mu,\phm,\ep}\colon\;\kx\lra \kappa[y],
$$
where $y$ is another indeterminate. We agree that $R(0)=0$.

%Let us recall the definition of the operator $R$. 

For a non-zero $f \in K[x]$ with $\phm$-expansion $f=\sum\nolimits_{0\le s}a_s\phm^s$, let us denote
\begin{equation}\label{sf}
S=S_{\ga}(N_{\mu,\phm}(f)),\qquad s(f)=s_{\ga}(f),\qquad s'(f)=s'_{\ga}(f).
\end{equation}

\begin{definition}\label{nrc}
The \emph{degree} of $S$ is the integer $d=(s'(f)-s(f))/e$. Denote
$$
s_j=s_0+je,\qquad \rc_j(f)=\hm(a_{s_j})\ep^{-j}\in\ggm^*, \qquad  0\le j\le d.
$$
These units $\rc_j(f)$ of degree $\mu(a_{s_j})+je\ga$ are called \emph{residual coefficients} of $f$.
\end{definition}

For any point $Q_s\in \cc$ as in (\ref{cloud}), we have
$$
Q_s\in S\ \sii\  s=s_j,\quad\deg(\rc_j(f))=\mu(f)-s_0\ga,\ \mbox{ for some }0\le j\le d.
$$

By construction,  $Q_{s_0}=Q_{s(f)}$ and $Q_{s_d}=Q_{s'(f)}$ lie on $S$, but for $0<j<d$, the point $Q_{s_j}\in\cc$ may lie strictly above $S$. %(See Figure \ref{figAlpha}) 
We define
$$
R(f)=\zeta_0+\zeta_1\,y+\cdots+\zeta_{d-1}y^{d-1}+y^d\in \kappa[y],
$$
with coefficients:
$$
\zeta_j=\begin{cases}
\rc_j(f)/\rc_d(f)\in \kappa^*,&\quad \mbox{ if }\ Q_{s_j}\mbox{ lies on }S,\\
0,&\quad \mbox{ otherwise}.
     \end{cases}
$$ 

%If $Q_{s_j}\in S$, the units $\rc_j(f)$, $\rc_d(f)$ have the same degree and $\zeta_j=\rc_j(f)/\rc_d(f)$ belongs to $\Delta^*=\kappa^*$.\e

\begin{figure}%[h]
\caption{Newton polygon $N_{\mu,\phm}(f)$. 
The line $L$ has slope $-\mu(\phm)$}\label{figAlpha}
\begin{center}
\setlength{\unitlength}{5mm}
\begin{picture}(14,8)
\put(2.75,4.85){$\bullet$}\put(3.5,4.4){$\times$}\put(4.5,3.9){$\times$}\put(5.5,3.4){$\times$}\put(6.4,2.95){$\times$}\put(7.4,2.45){$\times$}\put(8.35,2.05){$\bullet$}\put(6.5,4.25){$\bullet$}
\put(-1,0.5){\line(1,0){15}}\put(0,-0.5){\line(0,1){8}}
\put(-1,7){\line(2,-1){12}}
\put(2.9,5){\line(-1,2){1}}\put(2.9,5.04){\line(-1,2){1}}
\put(3,5){\line(2,-1){4.5}}\put(3,5.04){\line(2,-1){5.5}}
\put(8.5,2.2){\line(4,-1){3}}\put(8.5,2.24){\line(4,-1){3}}
\multiput(2.9,.4)(0,.25){19}{\vrule height2pt}
\multiput(8.55,.4)(0,.25){8}{\vrule height2pt}
%\multiput(1,.9)(0,.25){25}{\vrule height2pt}
\multiput(6.7,.4)(0,.25){16}{\vrule height2pt}
%\put(1.2,7.3){\begin{footnotesize}$P_0$\end{footnotesize}}
\put(8.45,2.6){\begin{footnotesize}$Q_{s_d}$\end{footnotesize}}
\put(2.8,5.4){\begin{footnotesize}$Q_{s_0}$\end{footnotesize}}
%\put(7,3.5){\begin{footnotesize}$P_j$\end{footnotesize}}
\put(6.4,5){\begin{footnotesize}$Q_{s_j}$\end{footnotesize}}
\put(6.5,-.3){\begin{footnotesize}$s_j$\end{footnotesize}}
\put(8.3,-.3){\begin{footnotesize}$s_d=s'(f)$\end{footnotesize}}
\put(.9,-.3){\begin{footnotesize}$s(f)=s_0$\end{footnotesize}}
\put(11.1,0.8){\begin{footnotesize}$L$\end{footnotesize}}
\put(-.5,-.3){\begin{footnotesize}$0$\end{footnotesize}}
%\put(-3.6,4.8){\begin{footnotesize}$u_0=u_r(f)$\end{footnotesize}}
%\multiput(-.1,5)(.25,0){13}{\hbox to 2pt{\hrulefill }}
\put(-.2,6.5){\line(1,0){.4}}
\put(-1.8,6.2){\begin{footnotesize}$\mu(f)$\end{footnotesize}}
\end{picture}
\end{center}
\end{figure}
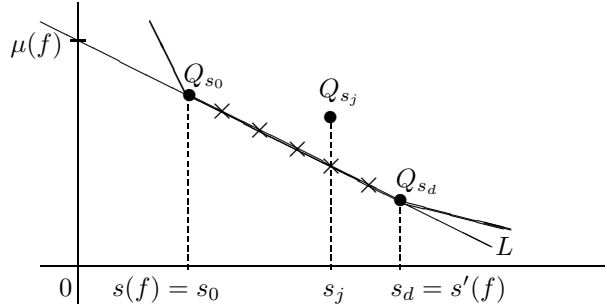

Let us emphasize that $\zeta_0$ is always nonzero. In other words,
\begin{equation}\label{basicR}
y\nmid R(f) ,\qquad \forall\,f\in\kx.
\end{equation}

The essential property of the operator $R$ is described in the next result.

\begin{theorem}\label{Hmug}
For  $\mu\in\Vkp$ and any $f\in\kx$, 
$$\hm(f)=\rc_d(f)\hm(\phm)^{s(f)}R(f)(\xi).$$
\end{theorem}

%Thus, the homogeneous element $\hm(f)$ splits into a product of a power of the prime $x_r$, times a unit, times a degree-zero element $R_r(f)(y_r)\in\Delta$.

%We now derive from Theorem \ref{Hmug} some more basic properties of the operator $R_r$.

\begin{corollary}\label{prodR} Let  $\mu\in\Vkp$ and $f,g\in\kx$. Then, 
\begin{enumerate}
 \item $R(fg)=R(f)R(g)$.
 \item $f \smu g\ \iff\ \rc_d(f)=\rc_d(g), \ s(f)=s(g) \ \mbox{ and }\ R(f)=R(g)$.
%\item $f \mmu g\ \iff\ s(f)\le s(g)\ \mbox{ and }\  R(f)\mid R(g) \ \mbox{ in }\  \kappa[y]$.
\end{enumerate}
\end{corollary}

The (non-canonical) operator $R$ is a down-to-earth representation of the canonical operator $\rr$, and it facilitates a characterization of key polynomials for $\mu$.
\begin{equation}\label{RR}
\rr(f)=\xi^{\lceil s(f)/e\rceil}R(f)(\xi)\Delta,\qquad \forall\,f\in\kx. 
\end{equation}

\begin{theorem}\label{charKP}
For $\mu\in\Vkp$, take $\phm\in\kpm$ of minimal degree $m$. 
A monic $f\in\kx$ is a key polynomial for $\mu$ if and  only if either
\begin{enumerate}
\item[(1)] $\deg(f)=m$ \,and\; $f\smu\phm$, in which case \ $\rr(f)=\xi\Delta$, or
\item[(2)] $\deg(f)=me\deg(R(f))$ \,and\; $R(f)$ is irreducible in $\kappa[y]$,
in which case 

$\rr(f)=R(f)(\xi)\Delta$  \;and\; $N_{\mu,\phm}(f)$ is one-sided of slope $-\mu(\phm)$.
\end{enumerate}
\end{theorem}

\begin{corollary}\label{efphi}
For $\mu\in\Vkp$, take $\phi\in\kpm$. Then, if $\phi\smu\phm$, we have
$$
e(v_{\mu,\phi}/v)=e(\mu)/\erel(\mu),\qquad f(v_{\mu,\phi}/v)=f(\mu),
$$
whereas in the case $\phi\not\smu\phm$, we get
$$
e(v_{\mu,\phi}/v)=e(\mu),\qquad 
f(v_{\mu,\phi}/v)=f(\mu)\deg(R(\phi)).$$
\end{corollary}

\begin{proof}
The statements in the case  $\phi\smu\phm$ follow from Propositions \ref{vphi} and \ref{maxideal}.

If  $\phi\not\smu\phm$, then $e(v_{\mu,\phi}/v)=e(\mu)$ by \cite[Cor. 6.4]{KeyPol}.

Also, $\rr(\phi)=R(\phi)(\xi)\Delta$ by Theorem \ref{charKP}. Hence, the statement about $f(v_{\mu,\phi}/v)$ follows from Proposition \ref{maxideal} and Theorem \ref{kstructure}:
$$
k_{\mu,\phi}\simeq \Delta/\rr(\phi)=\kappa[\xi]/\left(R(\phi)(\xi)\right)\simeq \kappa[y]/\left(R(\phi)\right), 
$$
so that $[k_{\mu,\phi}\colon \kappa]=\deg(R(\phi))$. Also, $[\kappa\colon k]=f(\mu)$ by definition.
\end{proof}

\begin{corollary}\label{samefiber}
Let $\phi,\phi'\in\kpm$. The following conditions are equivalent, and they imply \,$\deg(\phi)=\deg(\phi')$.
\begin{enumerate}
\item $\phi\smu\phi'$.
\item $\phi\mmu\phi'$. 
%\item $\hm(\phi)$ and $\hm(\phi')$ are associate primes in $\ggm$. 
\item $\rr(\phi)=\rr(\phi')$.
\item $R(\phi)=R(\phi')$.
\end{enumerate}
\end{corollary}

\section{Inductive valuations}\label{secIndVals}

Most of the results of this section are extracted from \cite{CompRP} and \cite{Vaq}.

\subsection{Augmentation of valuations}\label{subsecAugmentation}
Let $\mu$ be an extension of $v$ to $\kx$ admitting key polynomials.
Let $\iota\colon \gm\hookrightarrow \g'$ be an order-preserving embedding of $\gm$ into another abelian ordered group.

\begin{definition}\label{muprima}
Take $\phi\in \kpm$ and $\ga\in \g'$ any element such that $\mu(\phi)<\ga$. 

The \emph{augmented valuation} $\mu'=[\mu;\phi,\ga]$ is defined as follows on $\phi$-expansions: $$\mu'\colon\kx\rightarrow \g' \cup \left\{\infty\right\},\qquad \sum\nolimits_{0\le s}a_s\phi^s\longmapsto \mn\left\{\mu(a_s)+s\ga\mid 0\le s\right\}.$$
\end{definition}

There is a canonical homomorphism of graded algebras:
$$\ggm\lra\gg_{\mu'},\qquad \hm(f)\longmapsto
\begin{cases}\hmp(f),& \mbox{ if }\mu(f)=\mu'(f),\\ 0,& \mbox{  if }\mu(f)<\mu'(f).\end{cases}$$

\begin{proposition}\label{extension} Let $\mu'=[\mu;\phi,\ga]$, and let $f\in \kx$ be a non-zero polynomial.
\begin{enumerate}
\item[(a)] The valuation $\mu'$ extends $v$ and satisfies $\mu(f)\le \mu'(f)$. 

Equality holds if and only if $\phi\nmid_{\mu}f$. In this case,  $\hmp(f)$ is a unit in $\ggmp$.
%\item[(b)] The kernel of the homomorphism $\ggm\to\gg_{\mu'}$ is the prime ideal $\hm(\phi)\,\ggm$.
\item[(b)] The polynomial $\phi$ is a key polynomial for $\mu'$ of minimal degree, and $$N_{\mu',\phi}(f)=N_{\mu,\phi}(f).$$
\item[(c)] The value group of $\mu'$ is \ $\g_{\mu'}=\gen{\g_{\mu,\deg(\phi)},\ga}\subset\g'$. 
\item[(d)] $\k(\Delta_\mu\to \Delta_{\mu'})=\rr(\phi)$ \ and \ $ \im(\Delta_\mu\to \Delta_{\mu'})=\kappa(\mu')$.
\end{enumerate}
\end{proposition}

%The next result follows from Propositions \ref{extension} and \ref{maxideal}.
%\begin{equation}\label{imagedelta}
%\end{equation}

\begin{definition}\label{defproper}
A key polynomial $\phi$ for $\mu$ is said to be \emph{proper} if there exists some $\phm\in\kpm$ of minimal degree such that $\phi\not\smu \phm$. 
\end{definition}

\begin{lemma}\label{groupchain0}
Let $\phi$ be a proper key polynomial for $\mu$. Then,
\begin{enumerate}
\item  $\g_{v_{\mu,\phi}}=\g_{\mu,\deg(\phi)}=\gm$. 
\item All augmentations $\mu'=[\mu;\phi,\ga]$ have $\g_{\mu'}=\gen{\gm,\ga}\supset \gm$.
\end{enumerate}
\end{lemma}

%This result is false for improper key polynomials. However, we shall see in section \ref{subsecProper} that there is at most one $\mu$-equivalence class of improper key polynomials.

\subsection{The minimal extension of $v$ to $\kx$}\label{subsecDepth0}
Consider the ordered group $\left(\Z\times \qg\right)_{\op{lex}}$ with the lexicographical order, and the following valuation on $\kx$:%, which  is a kind of minimal extension of $v$ to $\kx$:
$$
\minf\colon \kx\lra \left(\Z\times \g\right)_{\op{lex}}\cup\{\infty\},\qquad 
f\longmapsto \left(-\deg(f),v\left(\lc(f)\right)\right)
$$
where $\lc(f)\in K^*$ is the leading coefficient of a non-zero polynomial $f$.

Since $\g_{\minf}=\left(\Z\times \g\right)_{\op{lex}}$, the extension $\minf/v$ is incommensurable.

The set of key polynomials of $\minf$ is
$$\op{KP}(\minf)=\left\{x+a\mid a\in K\right\},$$ 
and all these key polynomials are $\minf$-equivalent.

Let us fix an order-preserving embedding of abelian ordered groups:
$$%\begin{equation}\label{gembb}
\qg\hooklongrightarrow \left(\Z\times \qg\right)_{\op{lex}},\qquad \ga\longmapsto (0,\ga).
$$%\end{equation}
This  embedding allows the comparation of all $\mu\in\V$ with $\minf$. Clearly,
$$
\minf<\mu,\qquad \forall\,\mu\in\V.
$$

Consider augmentations of $\minf$  of the form:
$$
\mu_0(x+a,\ga):=[\minf;\,x+a,\,(0,\ga)],\qquad a\in K,\quad\ga\in\qg.
$$

%Lemma \ref{unique} shows under what conditions two of these augmentations coincide: $$\mu_0(x+a,\ga)=\mu_0(x+b,\ga_*)\ \iff\  v(a-b)\ge\ga=\ga_*.$$

By dropping the first (null) coordinate of all values of these valuations, we obtain equivalent valuations with values in $\qg$. We denote these valuations, which now belong to the space $\V$, with the same symbol $\mu_0(x+a,\ga)$.
They act as follows:
$$
\mu_0(x+a,\ga)\colon \ \sum\nolimits_{0\le s}a_s(x+a)^s\ \longmapsto\ \mn\left\{v(a_s)+s\ga\mid0\le s\right\},
$$
and their value group is $\g_{\mu_0(x+a,\ga)}=\gen{\g,\,\ga}$.

%For instance, the valuation $\mu_0(x,0)$ is known as \emph{Gauss' valuation}.

\subsection{MacLane chains}\label{subsecMLChains}

A valuation $\mu$ on $\kx$ is said to be \emph{inductive} if it is attained after a finite number of augmentation steps starting with the minimal valuation:
\begin{equation}\label{depth}
\minf\stackrel{\phi_0,\ga_0}\lra\  \mu_0\ \stackrel{\phi_1,\ga_1}\lra\  \mu_1\ \stackrel{\phi_2,\ga_2}\lra\ \cdots
\ \stackrel{\phi_{r-1},\ga_{r-1}}\lra\ \mu_{r-1} 
\ \stackrel{\phi_{r},\ga_{r}}\lra\ \mu_{r}=\mu,
\end{equation}
with values $\ga_0,\dots,\ga_r\in\qg$, and intermediate valuations  
$$\mu_0=\mu_0(\phi_0,\ga_0), \qquad\mu_i=[\mu_{i-1};\phi_i,\ga_i], \quad 1\leq i\leq r.
$$

We do not consider $\minf$ to be an inductive valuation; thus, inductive valuations belong to the space $\Vkp$. 
The integer $r\ge0$ is the \emph{length} of the chain (\ref{depth}).  

%Since the value group of an inductive valuation is a subgroup of $\qg$, all inductive valuations belong to our space $\V$.
%Let us denote the subset of all inductive valuations by $$\Vind:=\Vind(K,v)\subset\Vkp\subset\V.$$

By Proposition \ref{extension}, every $\phi_{i}$ is a key polynomial for $\mu_i$ of minimal degree. 
Hence, Theorem \ref{bound} shows that the sequence of weights $w(\mu_i)\in\qg$ grows strictly: 
$$%\begin{equation}\label{wrecurrence}
% \dfrac{\mu_{i}(\phi_{i+1})}{\deg(\phi_{i+1})}=
 w(\mu_{i})= \mu_{i}(\phi_i)/\deg(\phi_i)> \mu_{i-1}(\phi_{i})/\deg(\phi_{i})=w(\mu_{i-1}).
$$%\end{equation} 

Since $w(\mu_i)=\ga_i/\deg(\phi_i)$, the sequence $\;\ga_0<\dots<\ga_r\;$ grows strictly too.

\begin{lemma}\label{stable}
For a chain of augmentations as in {\rm (\ref{depth})}, consider $f\in \kx$ such that $\phi_i\nmid_{\mu_{i-1}}f$ for some $1\le i\le r$. Then, 
$\mu_{i-1}(f)=\mu_{i}(f)=\cdots=\mu_r(f)$. 
\end{lemma}

Since every $\phi_{i+1}$ is $\mu_i$-minimal, \cite[Prop. 3.7]{KeyPol} shows that
$$
1=\deg(\phi_0)\mid \deg(\phi_1)\mid\cdots\mid\deg(\phi_{r-1})\mid\deg(\phi_r).
$$

\begin{definition}
A \emph{MacLane chain} is a chain of augmentations (\ref{depth}) such that
$$
\phi_i\nmid_{\mu_{i-1}}\phi_{i-1},\qquad 0<i\le r.
$$
In particular, $\phi_i$ is a proper key polynomial for $\mu_{i-1}$ (cf. Definition \ref{defproper}).

An \emph{optimal MacLane chain} is a chain of augmentations (\ref{depth}) with  $$1=\deg(\phi_0)<\deg(\phi_1)<\cdots<\deg(\phi_r).$$
\end{definition}

Obviously, the truncation of a MacLane chain at the $i$-th node 
is a MacLane chain of the intermediate valuation $\mu_i$. 
By Lemma \ref{stable}, in a MacLane chain,
$$%\begin{equation}\label{muphii}
\mu(\phi_i)=\mu_i(\phi_i)=\ga_i,\qquad 0\le i\le r.
$$%\end{equation}

All inductive valuations admit optimal MacLane chains.
For these chains, there is a strong unicity statement.

\begin{proposition}\label{unicity}% \cite[Thms. 15.2, 15.3]{mcla}
Consider an optimal MacLane chain as in {\rm (\ref{depth})} and another optimal MacLane chain
$$
\minf\ \stackrel{\phi^*_0,\ga^*_0}\lra\  \mu^*_0\ \stackrel{\phi^*_1,\ga^*_1}\lra\ \cdots
\ \lra\ \mu^*_{t-1} 
\ \stackrel{\phi^*_{t},\ga^*_{t}}\lra\ \mu^*_{t}=\mu^*.
$$
Then, $\mu=\mu^*$ if and only if $r=t$ and
$$
\deg(\phi_i)=\deg(\phi^*_i), \quad \mu_{i-1}(\phi_i-\phi^*_i)\ge\ga_i=\ga^*_i \ \mbox{ for all }\ 0\le i \le r.
$$
In this case, we also have $\mu_i=\mu^*_i$ for all $0\le i \le r$. 
\end{proposition}

Therefore, in any optimal MacLane chain of an inductive valuation $\mu$ as in (\ref{depth}), the intermediate valuations $\mu_0,\dots,\mu_{r-1}$, the \emph{slopes} $\ga_0,\dots,\ga_r\in\qg$ and the positive integers $\deg(\phi_0),\dots,\deg(\phi_r)$, are intrinsic data of $\mu$.

\begin{definition}
The \emph{MacLane depth} of an inductive valuation $\mu$ is the length $r$ of any optimal MacLane chain of $\mu$.  
%Accordingly, inductive valuations are called \emph{finite-depth} valuations too. 
\end{definition}

A key polynomial $\phi$ for an inductive valuation $\mu$ is \emph{defectless} \cite{Vaq2},  \cite[Cor. 5.14]{CompRP}:
$$%\begin{equation}\label{defectind}
\deg(\phi)=e(v_{\mu,\phi}/v)f(v_{\mu,\phi}/v).
$$%\end{equation}

%\begin{lemma}\label{defect1phi}A key polynomial $\phi$ for an inductive valuation $\mu$  satisfies \end{lemma}

In particular, the valuation $v_{\mu,\phi}$ is the unique extension of $v$ to the field $K_\phi$.   
Therefore, whenever we deal with a key polynomial $\phi$ for an inductive valuation $\mu$, we shall use the following simplified notation:
$$
v_\phi=v_{\mu,\phi},\qquad k_\phi=k_{\mu,\phi},\qquad e(\phi)=e(v_\phi/v),\qquad f(\phi)=f(v_\phi/v).
$$

Take an optimal MacLane chain of $\mu$ as in {\rm (\ref{depth})}, and denote $\mu_{-1}=v$, $\kappa_{-1}=k$.

By Lemma \ref{groupchain0}, the value groups of the intermediate valuations form a chain:
$$\g_{v_{\phi_i}}=\g_{\mu_{i-1},\deg(\phi_i)}=\g_{\mu_{i-1}}\subset \g_{\mu_{i}}=\gen{\g_{\mu_{i-1}},\ga_i}, \quad 0\le i\le r.$$ 
Thus, every quotient $\g_{\mu_i}/\g_{\mu_{i-1}}$  
  is a finite cyclic group generated by $\ga_i$. 

The relative ramification indices are intrinsic invariants of $\mu$: 
$$
e_i:=\left(\g_{\mu_i}\colon \g_{\mu_{i-1}}\right)=\erel(\mu_i),\qquad 0\le i\le r.
$$

Also, the canonical homomorphisms $\gg_{\mu_{i-1}}\lra\gg_{\mu_i}$ induce a tower of fields
$$
 \kappa_0=\op{Im}( k\to\Delta_{\mu_0})\lra \cdots \lra  \kappa_i=\op{Im}(\Delta_{\mu_{i-1}}\to\Delta_{\mu_i})\lra \cdots \lra \kappa_r.
$$
By Propositions \ref{extension} and \ref{maxideal}, $ \kappa_i=\kappa(\mu_i)\simeq k_{\phi_i}$ for all $0\le i\le r$.

Let $R_i=R_{\mu_i,\phi_i,\ep_i}$ be the residual polynomial operators, determined by adequate units $\ep_i\in\gg_{\mu_i}^*$ of degree $-e_i\mu_i(\phi_i)$.

The relative residual degrees are intrinsic invariants of $\mu$ too: 
$$
f_{i-1}:=[\kappa_{i}\colon  \kappa_{i-1}]=\deg(R_{i-1}(\phi_i)),\qquad 0<i\le r,
$$ 
the last equality by Corollary \ref{efphi}.

Corollary \ref{efphi} provides as well a computation of the ramification index and residual degree of any $\phi\in\kpm$ such that $\phi\not\smu\phi_r$:
$$
e(\phi)=e(\mu)=e_0\cdots e_{r},\qquad f(\phi)=f(\mu)\deg(R_r(\phi))=f_0\cdots f_{r-1}\deg(R_r(\phi)).
$$

\section{Proper key polynomials and types}\label{secProperTypes}

\subsection{Proper key polynomials}\label{subsecProper}
%We keep dealing with our valued field $(K,v)$.

Let $\mu$ be an extension of $v$ to $\kx$ with $\kpm\ne\emptyset$.

A key polynomial $\phi$ for $\mu$ is proper if there exists a key polynomial for $\mu$, of minimal degree, which is not $\mu$-equivalent to $\phi$ (Definition \ref{defproper}). 

By Corollary \ref{samefiber}, properness is a property of $\mu$-equivalence classes of key polynomials for $\mu$.

If $\mu/v$ is an incommensurable extension, then there is only one $\mu$-equivalence class of key polynomials \cite[Thm. 5.3]{KeyPol}. Hence, this single class is improper.  \e

In this section, we study properness of key polynomials for valuations $\mu\in\Vkp$. 

By Theorem \ref{Max}, the set $\kpm/\!\smu$ is in canonical bijection with  $\mx(\Delta)$.  Hence, it makes sense to talk about proper maximal ideals in $\Delta$.

Let us fix some notation.

\begin{itemize}
\item $\phm$ \ key polynomial for $\mu$ of minimal degree $m$ in the set $\kpm$.
\item $e=\erel(\mu)=\left(\gm\colon\g_{\mu,m}\right)$ \ relative ramification index of $\mu$.
\item $R=R_{\mu,\phm,\ep}$ \ residual polynomial operator, for $\ep\in\ggm^*$  of degree $-e\mu(\phm)$.
\item $\xi=\hm(\phm)^e\ep\in\Delta$.
\item $[\phi]=[\phi]_\mu$ \ $\mu$-equivalence class of $\phi$ in the set $\kpm$.
\item $s_{\mu,\phi}(f)$ \ order with which $\hm(\phi)$ divides $\hm(f)$ in $\ggm$, for all $f\in\kx$.
\item $s(f)=s_{\mu,\phm}(f)$ \ notation coherent with that of (\ref{sf}), by Lemma \ref{length}.
\end{itemize}

Theorem \ref{charKP} shows that all key polynomials for $\mu$ have degree multiple of $m$:
\begin{equation}\label{degphi}
\deg(\phi)=\begin{cases}m, &\mbox{ if }\;\phi\in[\phm],\\
em\deg\left(R(\phi)\right),&\mbox{ if }\;\phi\not \in[\phm].
\end{cases}
\end{equation}

Let us show that properness of a key polynomial is determined by its degree.

\begin{lemma}\label{CritProp}
A key polynomial $\phi\in\kpm$ is proper if and only if \ $em\mid \deg(\phi)$.
\end{lemma}

\begin{proof}
If $\phi$ is proper,  then $em\mid \deg(\phi)$ by (\ref{degphi}).

Conversely, suppose that  $em\mid \deg(\phi)$.
If $e>1$, then (\ref{degphi}) shows that $\phi\not\in[\phm]$.   

If $e=1$, then $\gm=\g_{\mu,m}$, and there exists  $a\in\kx$ with $\deg(a)<m$ such that $\mu(a)=\mu(\phm)$. 
The polynomial $\phi'=\phm+a$ has $R(\phi')=y+\hm(a)\ep\in\kappa[y]$.

 By Theorem \ref{charKP}, $\phi'$ is another key polynomial of minimal degree, and $[\phi']\ne[\phm]$. Hence, either $\phi\not\in[\phi']$ or $\phi\not\in[\phm]$.
\end{proof}

The following result is an immediate consequence of (\ref{degphi}) and Lemma \ref{CritProp}.

\begin{corollary}\label{improper}\mbox{\null}
\begin{enumerate}
\item If $e=1$, all key polynomials for $\mu$ are proper.
\item If $e>1$, then $[\phm]$ is the only improper class. 

This class coincides with the set of all key polynomials of degree $m$. 
\end{enumerate}
\end{corollary}

%\noindent{\bf Example. }If $\mu$ is an inductive valuation and it admits a MacLane chain of length $r$ as in (\ref{depth}), then:\begin{enumerate}\item $m=m_r$, \ $e=e_r$.\item If $e_r>1$, the improper class of $\mu=\mu_r$ is $[\phi_r]$.\end{enumerate}\e

%\begin{corollary}\label{CritPropInd}Suppose that $\mu$ is inductive. Then, a key polynomial $\phi\in\kpm$ is proper if and only if there exists a MacLane chain of $\mu$ such that $\phi\not\smu\phi_r$, where $r$ is the length of the chain.\end{corollary}

\begin{definition}
We say that $f\in K[x]$ is \emph{$\mu$-proper} if either $f=0$, or $\phi\nmid_\mu f$ for all improper $\phi\in\kpm$.  
\end{definition}

The next result is an immediate consequence of Corollary \ref{improper}.

\begin{corollary}\label{improper2}
If $e=1$, all polynomials are $\mu$-proper.
 
If $e>1$, then $f\in\kx$ is $\mu$-improper if and only if $0<s(f)<\infty$. \end{corollary}

\begin{lemma}\label{multproper}
If at least one of the polynomials $g,h\in K[x]$ is $\mu$-proper, then $$\rr(gh)=\rr(g)\rr(h).$$ 
\end{lemma}

\begin{proof}
By Theorem \ref{kstructure} and (\ref{RR}), $\rr(gh)=\rr(g)\rr(h)$ is equivalent to: $$
y^{\lceil s(gh)/e\rceil}R(gh)=
y^{\lceil s(g)/e\rceil}R(g)
y^{\lceil s(h)/e\rceil}R(h).
$$
Since $s(gh)=s(g)+s(h)$ and  $R(gh)=R(g)R(h)$ (Corollary \ref{prodR}), this amounts to 
$$
\lceil (s(g)+s(h))/e\rceil=
\lceil s(g)/e\rceil+\lceil s(h)/e\rceil.
$$
If $e=1$ this equality is obvious. If $e>1$ it holds too, because either $s(g)=0$ or  $s(h)=0$, by Corollary \ref{improper2}. 
\end{proof}

\begin{proposition}\label{nextlength}
Let $\phi\in\kpm$ and $\ll=\rr(\phi)\in\mx(\Delta)$. For any $f\in K[x]$:
$$
\ord_{\ll}(\rr(f))=
\begin{cases}
s_{\mu,\phi}(f),&\quad\mbox{ if $\ll$ is proper},\\
\lceil s_{\mu,\phi}(f)/e\rceil,&\quad\mbox{ if $\ll$ is improper}. 
\end{cases}
$$
where $\ord_\ll(\rr(f))$ is the largest non-negative integer $n$ such that $\ll^n\mid \rr(f)$ in $\Delta$.
\end{proposition}

\begin{proof}
Let $\pset$ be a set of representatives of key polynomials under $\mu$-equivalence. For any non-zero $f\in K[x]$, there is a unique factorization:
$$
f\smu \prod\nolimits_{\phi\in \pset}\phi^{a_\phi},\qquad  a_\phi=s_{\mu,\phi}(f),\quad \forall\,\phi\in\pset,
$$
up to units in $\ggm$ \cite[Thm. 6.8]{KeyPol}.

If we apply $\rr$ to both terms of this factorization, Lemma \ref{multproper} shows that:
$$
\rr(f)=\rr\left(\prod\nolimits_{\phi\in \pset}\phi^{a_\phi}\right)=\prod\nolimits_{\phi\in \pset}\rr(\phi^{a_\phi}).
$$
For all proper $\phi\in\pset$ we have $\rr(\phi^{a_\phi})=\rr(\phi)^{a_\phi}$ by Lemma \ref{multproper}. 
Thus, $\rr(\phi)$ divides $\rr(f)$ with exponent $a_\phi$.

If $e>1$, the unique improper $\phi\in\pset$ satisfies $[\phi]=[\phm]$.  By Corollary \ref{samefiber}, $R(\phi)=R(\phm)=1$. Hence, (\ref{RR}) implies 
$$\rr(\phi)=\xi\Delta,\qquad \rr(\phi^{a_\phi})=\xi^{\lceil a_\phi/e\rceil}\Delta=\rr(\phi)^{\lceil a_\phi/e\rceil},
$$
so that $\rr(\phi)$ divides $\rr(f)$ with exponent $\lceil a_\phi/e\rceil$.
\end{proof}

\begin{corollary}\label{nextlength2}
Let $\phi\in\kpm$ such that $\phi\not\smu \phm$. Then, $R(\phi)\in\kappa[y]$ is a monic irreducible polynomial, and 
$$\ord_{R(\phi)}\left(R(f)\right)=s_{\mu,\phi}(f),\qquad \forall\,f\in K[x].$$ 
\end{corollary}

\begin{proof}
Denote $\psi=R(\phi)$ and $\ll=\rr(\phi)\in\mx\left(\Delta\right)$. 
By Theorem \ref{charKP} and (\ref{basicR}), $\psi$ is monic irreducible, $\psi\ne y$, and $\ll=\psi(\xi)\Delta$. 
By (\ref{RR}), $$\rr(f)=\xi^{\lceil s(f)/e\rceil}\,R(f)(\xi)\Delta.$$ Since $\psi\ne y$ and $\phi$ is proper, we get $\ord_\psi(R(f))=\ord_\ll(\rr(f))=s_{\mu,\phi}(f)$, by Theorem \ref{kstructure} and Proposition \ref{nextlength}. 
\end{proof}

\subsection{Valuations bounded by semivaluations}\label{subsecTotOrd}
%In this section, we collect a few auxiliary results which are still valid for an arbitrary valued field $(K,v)$ (not necessarily henselian). 

\begin{lemma}\label{minDegree}\cite[Thm. 1.15]{Vaq}
Let $\mu,\mu^*$ be a valuation and a semivaluation, respectively, which extend $v$ to $\kx$, and take values into a common ordered group.

Suppose that $\mu<\mu^*$, and let $\phi\in \kx$ be a monic polynomial with minimal degree satisfying $\mu(\phi)<\mu^*(\phi)$. Then, $\phi$ is a key polynomial for $\mu$ and $$\mu<[\mu;\phi,\ga]\le\mu^*,\quad \mbox{for }\ga=\mu^*(\phi).$$ Moreover, for any $f\in\kx$, the equality $\mu(f)=\mu^*(f)$ holds if and only if $\phi\nmid_\mu f$.
\end{lemma}

Actually, the subset $\Phi_{\mu,\mu^*}\subset\kpm$ of monic polynomials of minimal degree satisfying $\mu(\phi)<\mu^*(\phi)$ is a $\mu$-equivalence class of key polynomials.  

In fact, for any $\phi\in\Phi_{\mu,\mu^*}$,  $\phi'\in\kpm$, Lemma \ref{minDegree} and  Corollary \ref{samefiber} show that
$$
\phi'\in\Phi_{\mu,\mu^*}\ \sii\ \phi\mmu\phi'\ \sii\ \phi'\smu\phi.
$$

\begin{theorem}\label{totord}
Let $\mu^*$ be a $\qg$-valued semivaluation on $\kx$.
 Then, the interval $(\minf,\mu^*)\subset \V$ is totally ordered. 
\end{theorem}

\begin{proof}
Let $\eta,\eta'\in\V$ such that $\eta<\mu^*$, $\eta'<\mu^*$.
Take $\phi\in \Phi_{\eta,\mu^*}$, $\phi'\in \Phi_{\eta',\mu^*}$.

Suppose $\deg(\phi)<\deg(\phi')$. By the minimality of both degrees, 
$$\eta'(\phi)=\mu^*(\phi)>\eta(\phi);\quad \eta'(a)=\mu^*(a)=\eta(a),\ \forall \,a\in\kx,\mbox{ with }\deg(a)<\deg(\phi).$$

Hence, $\eta'>\eta$, because for any $g\in K[x]$ with $\phi$-expansion $g=\sum_{0\le s}a_s\phi^s$, we have:
$$
\eta'(g)\ge \mn\nolimits_{0 \le s}\{\eta'(a_s\phi^s)\}\ge\mn\nolimits_{0 \le s}\{\eta(a_s\phi^s)\}=\eta(g).
$$
 
Now, suppose $\deg(\phi)=\deg (\phi')$. Then, $\phi=\phi'+a$ for some $a\in K[x]$ with $\deg(a)<\deg(\phi)$.
By the $\eta'$-minimality of $\phi'$ we have $\eta'(\phi)\le\eta'(\phi')$. 

After eventually exchanging $\eta$ and $\eta'$, we may assume that
$\eta'(\phi')\le \eta(\phi)$. Then, 
$$
\eta'(\phi)\le\eta'(\phi')\le \eta(\phi)<\mu^*(\phi).
$$
Hence, $\phi\in\Phi_{\eta',\mu^*}$, and this implies $\phi'\sim_{\eta'}\phi$ by the remarks preceding this theorem.  Thus, $\eta'(\phi)=\eta'(\phi')\le \eta(\phi)$, and for any $\phi$-expansion  $g=\sum_{0\le s}a_s\phi^s\in\kx$, we get
$$
\eta(g)= \mn\nolimits_{0 \le s}\{\eta(a_s\phi^s)\}\ge\mn\nolimits_{0 \le s}\{\eta'(a_s\phi^s)\}=\eta'(g).
$$
Therefore, $\eta'\le \eta$.
\end{proof}

%%%%%%%%%%%%%%%%%%%%%%%%%%%%%%%%%%%%%%%%%%%%%%%%%%%%%%%%%%%%%%%%%%%%%%%%%%%%%%%%%%%%%%%%%%%%%%%%%%%%%%%%%%%%%%%%%%%%%%%%%%%%%%%%%%%%%%%%

%\section{Prime polynomials over henselian fields and key polynomials}\label{secPrimePols}

\subsection{Semivaluations attached to prime polynomials over henselian fields}\label{subsecvF}
\emph{From now on, we assume that the valued field $(K,v)$ is henselian. }

We still denote by $v$ its unique extension to a fixed algebraic closure $\kb$ of $K$.

%$$v\colon \kb\lra \qg\cup\{\infty\},$$

Let $\P=\P(K)$ be the set of all \emph{prime} (monic, irreducible) polynomials in $\kx$. 

For any  $F\in\P$ we denote 
\begin{itemize}
\item $K_F=\kx/(F\kx)$ \ finite extension of $K$ determined by $F$.
\item $\oo_F\supset\m_F$ \ valuation ring and maximal ideal of $v$ over $K_F$.
\item $ k_F=\oo_F/\m_F$ \ residue class field. 
\item $e(F)$, $f(F)$ \ ramification index and residual degree of $K_F/K$. 
%\item $f(F)=f()=[k_F\colon k]$ \ .
\item $d(F)=\deg(F)/e(F)f(F)$ \ defect of $F$.
\item $Z(F)$ \ multiset of roots of $F$ in $\kb$, counting multiplicities if $F$ is inseparable.
\end{itemize}

Let us consider the semivaluation:
$$
v_F\colon \kx \longtwoheadrightarrow K_F \stackrel{v}\lra \qg\cup\{\infty\},
$$
with support $v_F^{-1}(\infty)=F\kx$. 
By the henselian property,  
$$
v_F(f)=v(f(\t)),\qquad \forall\, f\in\kx,\ \,\forall\,\t\in Z(f).
$$

\begin {lemma}\label{symmetry}
For all $F,G\in\P$, we have \ $v_G(F)/\deg(F)=v_F(G)/\deg(G)$. 
\end {lemma}

\begin{proof} 
The lemma follows from 
$$\Res(F,G)=\prod\nolimits_{\t\in Z(F)}G(\t)=\pm\prod\nolimits_{\al\in Z(G)}F(\al),$$
since $v(G(\t))$, $v(F(\al)$ are constant for $\t\in Z(F)$, $\al\in Z(G)$, respectively. 
\end{proof}

Let $\mu$ be a valuation on $\kx$ admitting a key polynomial $\phi\in\kx$.

%Since $\mu$ and $v_F$ are $\qg$-valued, we may compare their values. We shall use the notation:$$\mu< v_F \sii \mu(f)\le v_F(f),\qquad \forall\,f\in\kx.$$

In section \ref{subsecGr}, we attached to $\phi$ a valuation $v_{\mu,\phi}$ on the field $K_\phi$. By the unicity of $v$, we must have $v_{\mu,\phi}=v$. Hence, the semivaluation on $\kx$ induced by  $v_{\mu,\phi}$ coincides with  the semivaluation $v_\phi$ determined by the prime polynomial $\phi\in\P$. 

More precisely,
for any $f\in\kx$ with $\phi$-expansion $f=\sum_{0\le s}a_s\phi^s$, 
$$
v_\phi(f)=v_{\mu,\phi}\left(f+\phi\kx\right)=\mu(a_0).
$$
%where $\al\in\kb$ is any choice of a root of $\phi$.

By Lemma \ref{minimal0}, $\mu<v_\phi$, regardless of the way we embedd $\gm$ into some group containing $\qg$. 
By Lemma \ref{minDegree}, for any non-zero $f\in\kx$ we have
\begin{equation}\label{mu<vphi}
\mu(f)=v_\phi(f)\ \sii\ \phi\nmid_\mu f.
\end{equation}

\subsection{Types over henselian fields}\label{secTypes}
In this section, we introduce \emph{types} and we generalize the results of \cite{ET} to henselian valued fields of arbitrary rank.

A \emph{type} is a pair $\ty=(\mu,\ll)$ belonging to the set
\begin{equation}
\T=\left\{(\mu,\ll) \mid \mu \in \Vkp, \ll \in \mx(\Delta_{\mu}),\ \ll \text{ proper}\right\}.
\end{equation}

A type encodes a certain set of prime polynomials.
A \emph{representative} of the type $\ty=(\mu,\ll)$ is a prime polynomial in the set 
\begin{equation}\nonumber
\rep(\ty)=\left\{ \phi \in \kpm \mid \rr_{\mu}(\phi) = \ll\right\}\subset \P.
\end{equation} 
By Theorem \ref{Max}, $\rep(\ty)$ is a proper equivalence class of key polynomials for $\mu$.

Any type $\ty\in\T$ determines a mapping
$$\ord_{\ty}\colon K[x]\,\lra\,\N, \qquad f  \,\longmapsto\, \ord_{\ll}(\rr_{\mu}(f))=s_{\mu,\phi}(f),\ \forall\,\phi\in\rep(\ty).$$
The last equality is a consequence of Proposition \ref{nextlength}.

\begin{lemma}\label{shareKP}
Suppose that $(K,v)$ is henselian, and let $\mu,\mu'$ be  two valuations admitting a common key polynomial $\phi\in\kpm\cap \op{KP}(\mu')$. Suppose that the value groups of $\mu,\mu'$ have been embedded into some common ordered group. Then,\begin{enumerate}
\item $\mu(\phi)=\mu'(\phi)\ \imp\ \mu=\mu'$.
\item $\mu(\phi)<\mu'(\phi)\ \imp\ \mu'=[\mu;\,\phi,\ga]$, with $\ga=\mu'(\phi)$.\end{enumerate}
\end{lemma}

\begin{proof}
Since $\mu,\mu'<v_\phi$, Theorem \ref{totord} shows that $\mu\le\mu'$, after eventually exchanging the valuations.
Moreover, (\ref{mu<vphi}) shows that
$$
\mu(a)=\mu'(a),\qquad \forall\,a\in\kx \mbox{ with }\deg(a)<\deg(\phi).
$$

If $\mu(\phi)=\mu'(\phi)$, then $\mu=\mu'$ because they coincide on $\phi$-expansions.
%This proves (1).

If $\mu(\phi)<\mu'(\phi)$, denote $\mu_*=[\mu;\,\phi,\mu'(\phi)]$. Then, $\phi$ is a common key polynomial for $\mu'$ and $\mu_*$. Since $\mu'(\phi)=\mu_*(\phi)$, item (1) shows that $\mu'=\mu_*$.
\end{proof}

%Let us show that types parameterize certain subsets of prime polynomials over henselian fields.

%The set $\rep(\ty)$ determines uniquely the type $\ty$.

\begin{theorem}\label{caractypes}
Suppose that $(K,v)$ is henselian. For all $\ty,\ty^{*}\in\T$ the following conditions are equivalent.
\begin{enumerate}
\item[(1)] $\ty=\ty^{*}$.
\item[(2)] $\ord_{\ty}=\ord_{\ty^{*}}$. 
\item[(3)] $\rep(\ty)=\rep(\ty^{*})$. 
\end{enumerate}
\end{theorem}

\begin{proof}
It is obvious that (1) implies (2). Let $\ty=(\mu,\ll),\;\ty^{*}=(\mu^*,\ll^*)$.

Let us show that (2) implies (3). 
Suppose that $\ord_\ty=\ord_{\ty^*}$. Take $\phi\in\rep(\ty)$, $\phi^*\in\rep(\ty^*)$. By Proposition \ref{nextlength}, 
\begin{equation}\label{phiphi}
1= \ord_{\ll}(\rr_{\mu}(\phi))= \ord_{\ll^{*}}(\rr_{\mu^{*}}(\phi))= s_{\mu^{*},\phi^{*}}(\phi).
\end{equation}
Thus, $\phi^{*} \mid_{\mu^{*}} \phi$. Since $\phi^{*}$ is $\mu^{*}$-minimal, we have $\deg(\phi) \ge \deg(\phi^{*})$. The symmetry of the argument implies $\deg(\phi^{*})=\deg(\phi)$. 

By Lemma \ref{mid=sim}, $\phi$ is a key polynomial for $\mu^{*}$, and $\phi^*\sim_{\mu^*}\phi$. 
Therefore, $$\rr_{\mu^*}(\phi)=\rr_{\mu^*}(\phi^*)=\ll^*,$$ so that $\phi$ is a representative of $\ty^*$.
This shows that $\rep(\ty)\subset \rep(\ty^*)$. 

By the symmetry of the argument, equality holds.\e

Finally, let us prove that (3) implies (1). Take any $\phi\in \rep(\ty)=\rep(\ty^{*})$. 

We need only to show that $\mu=\mu^*$, because then $\ll=\rrm(\phi)=\rr_{\mu^*}(\phi)=\ll^*$. 

Let us show that the assumption $\mu\ne\mu^*$ leads to a contradiction. By Lemma \ref{shareKP}, after eventually exchanging the valuations, this assumption leads to
\begin{equation}\label{mu*augm}
\mu^*=[\mu;\,\phi,\ga],\qquad \ga=\mu^*(\phi)>\mu(\phi).
\end{equation}
By Proposition \ref{extension}, $\phi$ is a key polynomial for $\mu^*$ of minimal degree. 

Since $\ll^*$ is proper, Corollary \ref{improper} shows that $\erel(\mu^*)=1$, and this implies $$\ga=\mu^*(\phi)\in\g_{\mu^*,\deg(\phi)}=\g_{\mu,\deg(\phi)}=\gm,$$
the last equalities by the definition of the augmented valuation and Lemma \ref{groupchain0}.

Thus, there exists $a\in\kx$ with $\deg(a)<\deg(\phi)$ and $\mu(a)=\ga>\mu(\phi)$. Therefore,
$$
\phi+a\smu \phi  \imp  \phi+a\in\rep(\ty),\qquad \phi+a\not\sim_{\mu^*} \phi  \imp  \phi+a\not\in\rep(\ty^*).
$$
This shows that $\rep(\ty)\ne \rep(\ty^*)$, against our assumption. 
\end{proof}

If we dropped properness of the maximal ideal $\ll$ from the definition of a type, there would be different types with the same sets of representatives.

On the other hand, we may still have types $\ty\ne\ty'$ with $\rep(\ty)\cap\rep(\ty')\ne\emptyset$. In order to avoid this situation, we introduce \emph{strong} types.

\begin{definition}
A key polynomial $\phi\in\kpm$ is said to be strong if $\deg(\phi)$ is strictly larger than the minimal degree of key polynomials for $\mu$. 

A maximal ideal $\ll\in\mx(\Delta_\mu)$ is strong if $\ll=\rrm(\phi)$ for some strong $\phi\in\kpm$. 

A type $\ty=(\mu,\ll)$ is strong if $\ll$ is strong.
\end{definition}

By Corollary \ref{improper}, a strong key polynomial is necessarily proper, and the converse is true if $\erel(\mu)>1$.
Let us denote by $\Tstr\subset\T$ the subset of strong types.

\begin{lemma}\label{strong=good}
Let $\ty,\ty'\in\Tstr$. If \ $\rep(\ty)\cap\rep(\ty')\ne\emptyset$, then $\ty=\ty'$.
\end{lemma}

\begin{proof}
Let $\ty=(\mu,\ll),\;\ty^{*}=(\mu^*,\ll^*)$, and take $\phi\in \rep(\ty)\cap\rep(\ty^{*})$.
The assumption $\ty,\ty'\in\Tstr$ implies that $\phi$ is a strong key polynomial simultaneously for $\mu$ and $\mu^*$.

Arguing as in the proof of Theorem \ref{caractypes}, we need only to show that $\mu=\mu^*$. 

By Lemma \ref{shareKP}, the assumption $\mu\ne\mu^*$ leads to (\ref{mu*augm}),
after eventually exchanging the valuations.
By Proposition \ref{extension}, $\phi$ is a key polynomial for $\mu^*$ of minimal degree. This contradicts the fact that $\phi$ is strong. 
\end{proof}

%\subsection*{Computational representation of types}Let $\ty=(\mu,\ll)$ be a type based on an inductive valuation $\mu$. From a computational perspective:\begin{itemize}\item $\mu$ is represented by an optimal MacLane chain.\item $\ll$ is represented by a monic, irreducible polynomial $\psi\in \kappa[y]$, $\psi\ne y$.\end{itemize}

%The data supported by the MacLane chain provide an explicit isomorphism $\kappa[y]\simeq \Delta_\mu$. Thus, $\psi$ determines a maximal ideal $\ll\in\mx(\Delta_\mu)$.  The properness of $\ll$ is guaranteed by the condition $\psi\ne y$. The type is strong if and only if $\erel(\mu)\deg(\psi)>1$. The techniques of \cite{CompRP} yield an algorithm to construct a representative of $\ty$.

\section{Key polynomials as $\mu$-factors of prime polynomials}\label{secKeyFactors}
We keep assuming that the valued field $(K,v)$ is henselian. 

In this section, we find out arithmetic properties of a prime polynomial $F\in\P$, derived from the existence of  a valuation $\mu\in\Vkp$ admitting a key polynomial $\phi$ such that $\phi\mmu F$.

\begin{theorem}\label{fundamental}
Let $F\in\P$, and let $\phi\in\kpm$ for some valuation $\mu\in\Vkp$. Then, $$\phi\mmu F \ \sii\ 
\mu(\phi)<v_F(\phi).$$ Moreover, if this condition holds, then:
\begin{enumerate}
\item[(1)] Either $F=\phi$, or the Newton polygon $N_{\mu,\phi}(F)$ is one-sided of slope $-v_F(\phi)$.    
\item[(2)] %The leading monomial of the $\phi$-expansion of $F$ is $\phi^\ell$, where $\ell=\ell(N_{\mu,\phi}(F))$.

$F\smu\phi^\ell$ with $\ell=\ell(N_{\mu,\phi}(F))=\deg(F)/\deg(\phi)$. 

In particular, $\rr(F)$ is a power of the maximal ideal $\rr(\phi)$ in $\Delta$.  
\end{enumerate}
\end{theorem}

\begin{proof}
If $F=\phi$ all statements of the theorem are trivial. Assume $F\ne\phi$.\e%, and let $\al\in\kb$ be a root of $\phi$. \e

If $\phi\nmid_\mu F$, then $\mu(F)=v_\phi(F)$ by (\ref{mu<vphi}). By Theorem \ref{bound} and Lemma \ref{symmetry},
$$
\mu(\phi)\ge \deg(\phi)\mu(F)/\deg(F)=\deg(\phi)v_\phi(F)/\deg(F)=v_F(\phi).
$$

Now, suppose $\phi\mmu F$. Let $\t\in\kb$ be a root of $F$, and consider the minimal polynomial of $\phi(\t)$ over $K$: 
$$g=b_0+b_1x+\cdots+b_kx^k\in K[x],\qquad b_k=1.$$ 
All roots of $g$ in $\kb$ have a constant $v$-value $\ga:=v(\phi(\t))=v_F(\phi)$. Hence,
\begin{equation}\label{slope}
v(b_0)=k\ga,\quad v(b_j)\ge (k-j)\ga, \ 1\le j<k,\quad v(b_k)=0.
\end{equation}
Let $G=\sum_{j=0}^kb_j\phi^j\in\kx$. By (\ref{slope}), $N_{\mu,\phi}(G)$ is one-sided of slope $-\ga$. Since $G(\t)=0$, the polynomial $F$ divides $G$ and Theorem \ref{product} shows that
\begin{equation}\label{sum}
N_{\mu,\phi}^{\mbox{\tiny pp}}(G)=N_{\mu,\phi}^{\mbox{\tiny pp}}(F)+N_{\mu,\phi}^{\mbox{\tiny pp}}(G/F).
\end{equation}
By Lemma \ref{length}, $\ell(N_{\mu,\phi}^{\mbox{\tiny pp}}(F))>0$, since $\phi\mmu F$. Thus, $N_{\mu,\phi}^{\mbox{\tiny pp}}(G)$ has positive length too. This implies $\mu(\phi)<\ga$ by the definition of the principal polygon (see section \ref{subsecNewton}). 

On the other hand, since $N_{\mu,\phi}(G)=N_{\mu,\phi}^{\mbox{\tiny pp}}(G)$ is one-sided of slope $-\ga$, (\ref{sum}) shows that  $N_{\mu,\phi}(F)$ is one-sided of slope $-\ga$ too. 

This proves that $\phi\mmu F$ if and only if $\mu(\phi)<\ga$, and that (1) holds in this case. \e

\begin{figure}
\caption{Newton polygon $N_{\mu,\phi}(F)$ when $\phi\mmu F$.  We take $\ga=v_F(\phi)$}\label{figNF}
\begin{center}
\setlength{\unitlength}{4mm}
\begin{picture}(10,6)
\put(9.8,0.25){$\bullet$}\put(-.25,3.75){$\bullet$}
\put(-1,.6){\line(1,0){14}}\put(0,-.4){\line(0,1){6}}
\put(0,4){\line(3,-1){10}}
%\multiput(9,.5)(0,.25){9}{\vrule height2pt}
%\multiput(-.1,2.6)(.25,0){37}{\hbox to 2pt{\hrulefill }}
\put(-4.7,3.8){\begin{footnotesize}$\mu(a_0)=\ell\ga$\end{footnotesize}}
\put(-.6,-.4){\begin{footnotesize}$0$\end{footnotesize}}
\put(9.8,-.5){\begin{footnotesize}$\ell$\end{footnotesize}}
\put(4.5,3){\begin{footnotesize}slope $-\ga$\end{footnotesize}}
\end{picture}
\end{center}
\end{figure}
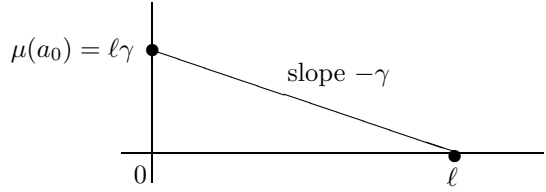

Let us prove item (2) in the case $\phi\mmu F$.
Consider the $\phi$-expansion $F=\sum_{s=0}^\ell a_s\phi^s$. 

By Lemma \ref{symmetry}, $$\mu(a_0)=v_\phi(F)=\deg(F)v_F(\phi)/\deg(\phi)=\deg(F)\ga/\deg(\phi).$$ 
Therefore, from the fact that $N_{\mu,\phi}(F)$ is one-sided of slope $-\ga$ we deduce: 
$$%\begin{align*}
 \mu(a_\ell)=\mu(a_0)-\ell\ga=\ga\,\dfrac{\deg(F)}{\deg(\phi)}-\ell\ga=\ga\,\dfrac{\deg(a_\ell)+\ell\deg(\phi)}{\deg(\phi)}-\ell\ga=\ga\,\dfrac{\deg(a_\ell)}{\deg(\phi)}.
$$%\end{align*}
If $\deg(a_\ell)>0$, then $a_\ell$ would be a monic polynomial contradicting Theorem \ref{bound}:
$$
\mu(a_\ell)/\deg(a_\ell)=\ga/\deg(\phi)>
\mu(\phi)/\deg(\phi).
$$ 
Hence, $a_\ell=1$, so that the leading monomial of the $\phi$-expansion of $F$ is $\phi^\ell$. In particular, $\ell=\deg(F)/\deg(\phi)$.

Since $\mu(\phi)<\ga$, we have $\mu(\phi^\ell)<\mu(a_s\phi^s)$ for all $s<\ell$. Thus, $F\smu\phi^\ell$. The statement about $\rr(F)=\rr(\phi^\ell)$ follows from Proposition \ref{nextlength}. 
\end{proof}

Our next aim is Theorem \ref{computation}, where we find another characterization of the property $\phi\mmu F$, which provides more arithmetic information on $F$. 

To this end, we need some auxiliary results.

\begin{lemma}\label{InsepSep}
For any $\beta\in\kb$ inseparable over $K$, and any $\rho\in\qg$, there exists $\beta\sep\in \kb$ separable over $K$ such that
$$\deg_K(\beta\sep)=\deg_K(\beta),\qquad v(\beta-\beta\sep)>\rho.$$ 
\end{lemma}

\begin{proof}
%If $\beta$ is separable over $K$, we may take $\beta\sep=\beta$.

Let $g\in\kx$ be the minimal polynomial of $\beta$ over $K$. We have $g'=0$.

Consider the polynomial $g\sep=g+\pi x\in\kx$, where $\pi\in K^*$ satisfies $$v(\pi)>\deg_K(\beta)\,\rho-v(\beta).$$ Since $g\sep'=\pi\ne0$, this polynomial is separable.  On the other hand,
$$
\sum_{\al\in\op{Z}(g\sep)}v(\beta-\al)=v\left(g\sep(\beta)\right)=v(\pi\beta)=v(\pi)+v(\beta)> \deg_K(\beta)\,\rho.
$$
Hence, there exists $\al\in\op{Z}(g\sep)$ such that $v(\beta-\al)>\rho$. We may take $\beta\sep=\al$.
\end{proof}

\begin{lemma}\label{averageToSingle}
Let $F\in\P$, $\mu\in\Vkp$ and $\phi\in\kpm$ such that $\phi\mmu F$.
Let $\t\in Z(F)$. 

Then, for all $\beta\in\kb$ with $\deg_K(\beta)<\deg(\phi)$, we have $$v(\t-\beta)<\mx\left\{v(\t-\al)\mid \al\in Z(\phi)\right\}.$$
\end{lemma}

\begin{proof}
Denote $\delta=\mx\left\{v(\t-\al)\mid \al\in Z(\phi)\right\}$, and fix $\al\in Z(\phi)$ with $v(\t-\al)=\delta$. 

Let $\beta\in\kb$ with $\deg_K(\beta)<\deg(\phi)$. We want to show that 
 $v(\t-\beta)<\delta$, or equivalently,  $v(\al-\beta)<\delta$. 
 
% We may assume that $$v(\al-\beta)=\mx\{v(\al-\beta')\mid\beta'\in\op{Z}(g)\}.$$

By Lemma \ref{InsepSep}, we may assume that $\t$, $\al$ and $\beta$ are separable over $K$.
Consider a finite Galois extension  $M/K$ containing $\t$, $\al$ and $\beta$, and denote $G=\gal(M/K)$. 

Let us assume that $v(\al-\beta)\ge\delta$, and show that this leads to a contradiction. 

For all $\sigma\in G$ we get:
$$%\begin{equation}\label{inici}
\as{1.3}
\begin{array}{rl}
v(\al-\sigma(\beta))&=\ \, v(\al-\t+\t-\sigma(\al)+\sigma(\al)-\sigma(\beta))\\
&\ge\ \,\mn\{v(\al-\t),\,v(\t-\sigma(\al)),\,v(\sigma(\al)-\sigma(\beta))\}=v(\t-\sigma(\al)),
\end{array}
$$%\end{equation}
because $v(\t-\sigma(\al))\le\delta=v(\al-\t)$ and $v(\sigma(\al)-\sigma(\beta))=v(\al-\beta)\ge\delta$. 

Therefore, if $g\in\kx$ is the minimal polynomial of $\beta$ over $K$, we get
$$%\begin{equation}\label{compare}
\dfrac{\#G}{\deg(g)}\,v(g(\al))=\sum_{\sigma\in G}v(\al-\sigma(\beta))\ge\sum_{\sigma\in G}v(\t-\sigma(\al))= \dfrac{\#G}{\deg(\phi)}\,v(\phi(\t)).
$$%\end{equation}

By Theorem \ref{fundamental}, this inequality implies
$$
\mu(g)/\deg(g)=v(g(\al))/\deg(g)\ge v(\phi(\t))/\deg(\phi)>\mu(\phi)/\deg(\phi),
$$
which contradicts Theorem \ref{bound}.
\end{proof}

%It is well known that the inequality of Lemma \ref{averageToSingle} leads to $g(\t)\sim_v g(\al)$ \cite{Ok}.
%Let us reproduce an argument of \cite{BhaKha}.

\begin{proposition}\label{samevalue}
Let $F\in\P$, $\mu\in\Vkp$ and $\phi\in\kpm$ such that $\phi\mmu F$. %Choose arbitrary roots $\t\in Z(F)$, $\al\in Z(\phi)$.

Then, for all $g\in\kx$ with $\deg(g)<\deg(\phi)$, we have $v_\phi(g)=v_F(g)$.
\end{proposition}

\begin{proof}
Take $\t\in Z(F)$ and $\al\in Z(\phi)$ such that $v(\t-\al)=\mx\left\{v(\t-\al'\mid \al'\in Z(\phi)\right\}$. 

Write $g=c\prod_j(x-\beta_j)$ with $c\in K$, $\be_j\in\kb$. Since $\deg_K(\beta_j)<\deg(\phi)$,  Lemma \ref{averageToSingle} shows that $\t-\beta_j\sim_v \al-\beta_j$ for all $j$. Therefore, $g(\t)\sim_v g(\al)$. 

In particular, $v(g(\t))=v(g(\al))$. 
\end{proof}

\begin{theorem}\label{computation}
 Let $F\in \P$. A valuation $\mu\in\Vkp$ admits a key polynomial $\phi$ such that  $\phi\mmu F$ if and ony if $\mu< v_F $. 
  In this case, for all non-zero $f\in K[x]$,
\begin{equation}\label{equality}
\mu(f)=v_F(f) \ \sii\ \phi\nmid_{\mu}f. 
\end{equation}
\end{theorem}

\begin{proof}
By Lemma \ref{minDegree}, the condition $\mu< v_F$ implies the existence of $\phi\in\kpm$ satisfying (\ref{equality}). In particular,  $\phi\mmu F$. %\e

Conversely, suppose that $\phi\mmu F$ for some $\phi\in\kpm$. By Proposition \ref{samevalue}, for any $a\in\kx$ with $\deg(a)<\deg(\phi)$ we have $v_F(a)=v_\phi(a)=\mu(a)$. 

On the other hand, $v_F(\phi)>\mu(\phi)$ by Theorem \ref{fundamental}. Hence, for any $f\in \kx$ with $\phi$-expansion $f=\sum_{0\le s}a_s\phi^s$, we have
$$
v_F(f)\ge \mn\left\{v_F(a_s\phi^s)\mid 0\le s\right\}\ge\mn\left\{\mu\left(a_s\phi^s\right)\mid 0\le s\right\}=\mu(f).
$$

Hence, $\mu<v_F$, and Lemma \ref{minDegree} shows that $\phi$ satisfies (\ref{equality}).
\end{proof}

%\noindent{\bf Remark.}
Theorem \ref{computation} provides a practical device for the computation of $v_F$.
For any given $f\in K[x]$, one seeks a pair $\mu,\phi$ such that $\phi\mmu F$ and $\phi\nmid_\mu f$, leading to $v_F(f)=\mu(f)$.

This yields a very efficient routine for the computation of the valuations attached to prime ideals in number fields, or places of function fields of curves \cite{newapp,gen}.

\begin{corollary}\label{efphiF}
Let $F\in\P$, $\mu\in\Vkp$ and $\phi\in\kpm$ such that $\phi\mmu F$. Then,
$$
e(\phi)\mid e(F),\qquad f(\phi)\mid f(F),\qquad d(\phi)\mid d(F).
$$
\end{corollary}

\begin{proof}
Take arbitrary roots $\al\in Z(\phi)$, $\t\in Z(F)$.

The elements of $\g_{v_\phi}$ are of the form $v(g(\al))$ for $g\in\kx$ with $\deg(g)<\deg(\phi)$.  
By Proposition \ref{samevalue}, these values coincide with $v(g(\t))\in\g_{v_F}$. This proves $e(\phi)\mid e(F)$.

By Theorem \ref{computation}, $\mu<v_F$, and this determines a canonical ring homomorphism
$$\Delta_\mu\lra  k_F,\qquad g+\pset_0^+(\mu)\ \longmapsto\ g(\t)+\m_F.
$$ 
The kernel $\ll_F$ of this homomorphism is a non-zero prime ideal of  $\Delta_\mu$. Since this ring is a PID, $\ll_F$ is a maximal ideal. 

Clearly, $\rrm(F)\subset \ll_F$. By Theorem \ref{fundamental}, there exists a positive integer $n$ such that $\rr_{\mu}(\phi)^n=\rr_{\mu}(F)\subset \ll_F$. Thus, $\rr_{\mu}(\phi)=\ll_F$, because both are maximal ideals. 

By Proposition \ref{maxideal}, $ k_\phi \simeq \Delta_\mu/\rr_{\mu}(\phi)= \Delta_\mu/\ll_F$, which is isomorphic to a subfield of $ k_F$. This proves $f(\phi)\mid f(F)$. 

The fact that $d(\phi)\mid d(F)$ follows from \cite{Vaq2}.
\end{proof}

If $\phi\mmu F$, we may think $\phi$ as a ``germ" of an approximation to $F$, and the value $v_F(\phi)$ as a measure of the quality of the approximation.

By Lemma \ref{minDegree}, the valuation $\mu'=[\mu;\,\phi,v_F(\phi)]$ is closer to $v_F$, and it is natural to expect that any $\phi'\in\op{KP}(\mu')$, such that $\phi'\mid_{\mu'}F$, will be a better approximation to $F$.

The next result is a crucial step for an efficient computation of these approximations. The prime polynomial $F$ is usually unknown in practice, but the key polynomial $\phi'$ may be constructed by using only some ``residual" information about $F$.

\begin{lemma}\label{goingup}
Take $F\in\P$, $\mu\in\Vkp$, $\phi\in\kpm$ such that $\phi\mmu F$ and $\phi\ne F$. 

For $\mu'=[\mu;\,\phi,v_F(\phi)]$, let $\kappa'=\kappa(\mu')$ and $R'=R_{\mu',\phi,\ep'}$ be the residual polynomial operator (section \ref{subsecRi}), for some $\ep'\in\gg_{\mu'}^*$ of degree $-e'v_F(\phi)$, where $e'=\erel(\mu')$. 

Then, $R'(F)\in\kappa'[y]$ is the power of some monic irreducible polynomial $\psi\in\kappa'[y]$.

Moreover, take any   monic polynomial $\phi'\in\kx$ of degree $e'\deg(\psi)\deg(\phi)$ such that $R'(\phi')=\psi$. 
Then, $\phi'$ is a proper key polynomial for $\mu'$ such that $\phi'\mid_{\mu'}F$, and
$$
v_F(\phi)<v_F(\phi'),\qquad e(\phi')=e(\phi)\,e',\qquad f(\phi')=f(\phi)\deg(\psi),\qquad d(\phi')=d(\phi).
$$
\end{lemma}

\begin{proof}
By Proposition \ref{extension} and Theorem \ref{fundamental}, $N_{\mu',\phi}(F)=N_{\mu,\phi}(F)$ is one-sided of slope $-v_F(\phi)=-\mu'(\phi)$, and length $\ell=\deg(F)/\deg(\phi)$. 

By the definition of $R'$, $\deg(R'(F))=\ell/e'>0$ and $y\nmid R'(F)$. 
Let $\psi$ be a monic irreducible factor of $R'(F)$ in $ \kappa'[y]$. 

Take a monic $\phi'\in\kx$ such that $\deg(\phi')=e'\deg(\psi)\deg(\phi)$ and $R'(\phi')=\psi$.

Since $\phi$ is a key polynomial for $\mu'$ of minimal degree,  Theorem \ref{charKP} shows that $\phi'\not\sim_{\mu'}\phi$. In particular, $\phi'$ is a proper key polynomial for $\mu'$.

By Corollary \ref{nextlength2}, $s_{\mu',\phi'}(F)=\ord_{\psi}(R'(F))>0$, so that $\phi'\mid_{\mu'}F$. 

By using Theorems \ref{fundamental} and \ref{bound}, we deduce
$$
v_F(\phi')>\mu'(\phi')=\deg(\phi')\,\mu'(\phi)/\deg(\phi)\ge \mu'(\phi)=v_F(\phi),
$$
and $F\sim_{\mu'}(\phi')^{\ell'}$, for $\ell'=\deg(F)/\deg(\phi')$. 
By Corollary \ref{prodR}, $R'(F)=R'(\phi')^{\ell'}=\psi^{\ell'}$.

Now, by Corollary \ref{efphi}, applied to $\phi$ and $\phi'$ as key polynomials for $\mu'$, we get
$$
e(\phi')=e(\mu')=e'e(\phi),\qquad f(\phi')=f(\mu')\deg(\psi)=f(\phi)\deg(\psi).
$$

Finally, $d(\phi')=d(\phi)$ follows from $\deg(\phi')=e'\deg(\psi)\deg(\phi)=e(\phi')f(\phi')d(\phi)$. 
\end{proof}

\subsection*{A generalization of Hensel's lemma}

Theorems \ref{fundamental} and \ref{computation} provide a fundamental result concerning factorization of polynomials over henselian fields. %It has to be considered as a vast generalization of Hensel's lemma. 

%Let us introduce some useful notation.

For $\mu\in\Vkp$, $\phi\in\kpm$ and $\ga\in\qg$ such that $\ga>\mu(\phi)$, let us denote
\begin{itemize}
\item $\mu_\ga=[\mu;\phi, \ga]$.
\item $e_\ga=\erel(\mu_\ga)$ \ relative ramification index of $\mu_\ga$. 
\item $\kappa_\ga=\kappa(\mu_\ga)$ \ algebraic closure of $k$ in $\Delta_{\mu_\ga}$.
\item $R_{\mu_\ga}\colon \kx\to \kappa_\ga[y]$ \ operator  $R_{\mu_\ga,\phi,\ep_\ga}$, for some $\ep_\ga\in\gg_{\mu_\ga}^*$ of degree $-e_\ga\ga$. 
\end{itemize}

\begin{theorem}\label{main}
Let $\phi\in\kpm$ for some $\mu\in\Vkp$. Let  $f\in \kx$ be a 
 monic polynomial.
 For each slope $-\ga$ of the  principal Newton polygon $N_{\mu,\phi}^{\mbox{\tiny pp}}(f)$, let
$$
R_{\mu_\ga}(f)=\prod\nolimits_\psi \psi^{a_{\ga,\psi}}
$$
be the factorization of $R_{\mu_\ga}(f)$ into a product of powers of pairwise different monic irreducible polynomials $\psi\in \kappa_\ga[y]$.

Then, $f$ factorizes  in $\kx$ into a product of monic polynomials:
\begin{equation}\label{factors}
f=f_0\,\phi^{\ord_\phi(f)}\prod\nolimits_{(\ga,\psi)} f_{\ga,\psi}.
\end{equation}
%where $-\ga$ runs on the slopes of  $N_{\mu,\phi}^{\mbox{\tiny pp}}(f)$, and for each $\ga$, $\psi$ runs on the monic irreducible factors of $R_{\mu_\ga}(f)$ in $\kappa_\ga[y]$.  

If $\ell=\ell\left(N^{\mbox{\tiny pp}}_{\mu,\phi}(f)\right)$, the degrees of the factors are:
$$\deg(f_0)=\deg(f)-\ell\deg(\phi),\qquad
\deg(f_{\ga,\psi})=e_\ga a_{\ga,\psi}\deg(\psi)\deg(\phi).
$$ 
Moreover, for any pair $(\ga,\psi)$, it holds:
\begin{enumerate}
\item $N_{\mu,\phi}(f_{\ga,\psi})$ is one-sided of slope $-\ga$ and length $e_\ga a_{\ga,\psi}\deg(\psi)$. 
\item $R_{\mu_\ga}(f_{\ga,\psi})=\psi^{a_{\ga,\psi}}$.
\item For every prime factor $F$ of $f_{\ga,\psi}$, we have $v_F(\phi)=\ga$ and
$$
e(\phi)\,e_\ga\mid e(F),\qquad f(\phi)\deg(\psi)\mid f(F),\qquad d(\phi)\mid d(F).
$$
\item If $a_{\ga,\psi}=1$, then $f_{\ga,\psi}$ is irreducible and it is a key polynomial for $\mu_\ga$, with
$$
e(f_{\ga,\psi})=e(\phi)\,e_\ga,\qquad f(f_{\ga,\psi})=f(\phi)\deg(\psi),\qquad d(f_{\ga,\psi})=d(\phi).
$$
\end{enumerate}
\end{theorem}
 
\begin{proof}
Let $f=F_1\cdots F_t$ be the factorization of $f$ into a product of prime polynomials in $\kx$. These prime factors are not necessarily pairwise different.

Let us group these prime factors according to some of their properties with respect to the pair $\mu,\phi$. 
\begin{itemize}
\item The factor $f_0$ is the product of all $F_j$ satisfying $\phi\nmid_\mu F_j$.
\item The factor $\phi^{\ord_\phi(f)}$ is the product of all $F_j$ equal to $\phi$. 
\item The factor $f_{\ga,\psi}$ is the product of all $F_j$ such that $\phi\mmu F_j$, $N_{\mu,\phi}(F_j)$ is one-sided of slope $-\ga$ and $R_{\mu_\ga}(F_j)$ is a power of $\psi$.
\end{itemize}

By Theorem \ref{fundamental}, the factors $F_j\ne\phi$ such that $\phi\mmu F_j$ have one-sided Newton polygon $N_{\mu,\phi}(F_j)$ of slope $-\ga$, with $\ga=v_{F_j}(\phi)>\mu(\phi)$. By Lemma \ref{goingup}, $R_{\mu_\ga}(F_j)$ is a power of some irreducible $\psi\in\kappa_\ga[y]$. 

By Theorem \ref{product}, $-\ga$ is one of the slopes of $N^{\mbox{\tiny pp}}_{\mu, \phi}(f)$, and by Corollary \ref{prodR}, $\psi$ is one of the irreducible factors of $R_{\mu_\ga}(f)$.
Therefore, every prime factor $F_j$ such that $F_j\ne\phi$ and $\phi\mmu F_j$ falls into one (and only one) of the factors $f_{\ga,\psi}$. 

This proves the factorization (\ref{factors}).%\e

%Let us now show a crucial fact; namely, we are able to identify the degrees of the resulting factors of $f$ in terms of computable discrete data. 

By Lemma \ref{length} and Theorem \ref{fundamental}, for all $1\le j\le t$ we have 
$$\ell_j:=\ell\left(N^{\mbox{\tiny pp}}_{\mu,\phi}(F_j)\right)=s_{\mu,\phi}(F_j)=\begin{cases}0,& \mbox{ if }\phi\nmid_\mu F_j,\\                                                                    \deg(F_j)/\deg(\phi),&\mbox{ if } \phi\mmu F_j.
\end{cases}
$$ 
By Theorem \ref{product}, $\ell=\ell_1+\cdots+\ell_t$. Hence,
$$%\begin{align*}
 \deg(f)-\deg(f_0)=\sum_{\phi\mmu F_j}\deg(F_j)=\sum_{\phi\mmu F_j}\ell_j\deg(\phi)=\sum_{j=1}^t\ell_j\deg(\phi)=\ell\deg(\phi).
$$%\end{align*}

Now, for each pair $(\ga,\psi)$, let $J(\ga,\psi)$ be the set of indices $j$ such that $F_j$ is a prime factor of $f_{\ga,\psi}$. 
For all $j\in J(\ga,\psi)$, %the Newton polygon $N_{\mu,\phi}(F_j)=N^{\mbox{\tiny pp}}_{\mu,\phi}(F_j)$ is one-sided of slope $-\ga$, and $R_{\mu_\ga}(F_j)$ is a power of $\psi$. Hence,
\begin{equation}\label{lj}
\ell_j=\ell\left(N_{\mu,\phi}(F_j)\right)=e_\ga\deg\left(R_{\mu_\ga}(F_j)\right)=e_\ga\deg(\psi)\ord_\psi R_{\mu_\ga}(F_j).
\end{equation}

On the other hand, $R_{\mu_\ga}(F_j)\in \kappa_\mu^*$ is a constant for all $j\not\in  J(\ga,\psi)$. In fact, $R_{\mu_\ga}(\phi)=1$, and for $\phi\nmid_\mu F_j$, the condition $\ell\left(N^{\mbox{\tiny pp}}_{\mu,\phi}(F_j)\right)=0$  implies that the $\ga$-component of $N_{\mu,\phi}(F_j)$ is reduced to a single point (cf. section \ref{subsecNewton}).

Corollary \ref{prodR} shows that $R_{\mu_\ga}(f_{\ga,\psi})=\prod\nolimits_{j\in J(\ga,\psi)}R_{\mu_\ga}(F_j)$. Hence,
\begin{equation}\label{apsi}
a_{\ga,\psi}=\ord_\psi R_{\mu_\ga}(f)=\ord_\psi R_{\mu_\ga}(f_{\ga,\psi})=\sum\nolimits_{j\in J(\ga,\psi)}\ord_\psi R_{\mu_\ga}(F_j), 
\end{equation}
We may conclude that
$$
\deg(f_{\ga,\psi})=\sum\nolimits_{j\in J(\ga,\psi)}\deg(F_j)=\sum\nolimits_{j\in J(\ga,\psi)}\ell_j\deg(\phi)=e_\ga a_{\ga,\psi}\deg(\psi)\deg(\phi).
$$

Moreover,  items (1) and (2) follow from Theorem \ref{product} and Corollary \ref{prodR}, having in mind equations (\ref{lj}) and (\ref{apsi}).

Consider a monic polynomial $\phi_{\ga,\psi}\in\kx$ of degree $e_\ga\deg(\psi)\deg(\phi)$ such that $R_{\mu_\ga}(\phi_{\ga,\psi})=\psi$.
By Lemma \ref{goingup}, $\phi_{\ga,\psi}$ is a key polynomial for $\mu_\ga$ such that $\phi_{\ga,\psi}\mid_{\mu_\ga} F$ for each prime factor of $f_{\ga,\psi}$ and
$$
e(\phi_{\ga,\psi})=e_\ga e(\phi),\qquad 
f(\phi_{\ga,\psi})=\deg(\psi) f(\phi),\qquad d(\phi_{\ga,\psi})=d(\phi).
$$
Thus, item (3) follows from Corollary \ref{efphiF}.

If $a_{\ga,\psi}=1$, we may take $\phi_{\ga,\psi}=f_{\ga,\psi}$, by items (1) and (2). This proves item (4). 
\end{proof}

\noindent{\bf Remarks. }
(1)  Theorem \ref{main} has been recently found by Jakhr-Khanduja-Sangwan, with a slightly different formulation  \cite{J-Khan}. The authors use the language of minimal pairs of definition of residually transcendental valuations.

Our proof is shorter and more accessible, mainly because the residual polynomial operator is a more malleable tool than the classical technique of \emph{lifting} polynomials from $\kappa[y]$ to $K[x]$.\e

(2)  This result is valid for an arbitrary valued field $(K,v)$, as long as the valuation $\mu$ is inductive. In this case, $\mu$ may be lifted to the henselization $K^h$ of $(K,v)$, and  $\phi$ is still a key polynomial of the lifted valuation.

In this way, it may be used to detect information about the prime factors in $K^h[x]$ of any given $f\in\kx$.

\section{Defectless polynomials and inductive valuations}\label{secAppr}
We keep assuming that $(K,v)$ is a henselian valued field.

Let $F\in\P$ be a prime polynomial, and let $\mu\in\Vkp$ admitting a key polynomial $\phi$ such that $\phi\mmu F$.
We may think of $\phi$ as a germ of an approximation to $F$. 

The iteration of the procedure of Lemma \ref{goingup} yields a MacLane chain based on the initial valuation $\mu$, with strictly better approximations:
$$
\mu\stackrel{\phi,\ga}\lra 
\mu'\stackrel{\phi',\ga'}\lra 
\mu''\stackrel{\phi'',\ga''}\lra \cdots ,\qquad \ga=v_F(\phi)<\ga'=v_F(\phi')<\ga''=v_F(\phi'')<\,\cdots
$$
%The initial key polynomial $\phi$ is not necessarily proper.

We say that this process \emph{converges} to $ F$ if after a finite number of steps we reach a valuation $\mu^{(n)}$ such that $F$ is a key polynomial for $\mu^{(n)}$. 

%Since all key polynomials have $\deg(\phi^{(n)})\le \deg(F)$, the value $\deg(\phi^{(n)})$ estabilizes. 

Since $\phi^{(n)}\mid_{\mu^{(n)}}F$, Lemma \ref{mid=sim} and Corollary \ref{samefiber} show that the process converges if and only if we reach a key polynomial with $\deg\left(\phi^{(n)}\right)=\deg(F)$.

In this section, we discuss this convergence when $F$ is defectless; that is, when it has trivial defect $d(F)=1$.

\subsection{Okutsu frames of defectless polynomials}\label{subsecOk=Dless}
Consider a prime polynomial $F\in\P$ of degree $n>1$, and a fixed root $\t\in Z(F)\subset \kb$.

For any integer $1<m\le n$, consider the set of weighted values
$$
W_m(F)=\left\{\dfrac{v(g(\t))}{\deg(g)}\;\Big|\; g\in\kx\mbox{ monic},\ 0<\deg(g)<m\right\}\subset\qg.
$$

\begin{definition}\label{minpairPoly}
Suppose that the set $W_n(F)$ contains a maximal value:
$$w(F):=\mx\left(W_n(F)\right).$$ 

We say that $\phi,F$ is a \emph{distinguished pair} of polynomials if $\phi\in\kx$ is a monic polynomial of minimal degree among the monic polynomials satisfying
$$
0<\deg(\phi)<n, \quad\ v(\phi(\t))/\deg(\phi)=w(F).
$$
\end{definition}

\begin{definition}
We say that $F$ is an \emph{Okutsu polynomial} if all sets $W_m(F)$ contain a maximal element, for $1<m\le n$. 
\end{definition}

Suppose that $F$ is an Okutsu polynomial, and  $\phi,F$ is a distinguished pair. 

If $\deg(\phi)>1$, we may consider a monic $\phi'\in\kx$ of minimal degree such that  
$$
0<\deg(\phi')<\deg(\phi),\qquad \dfrac{v(\phi'(\t))}{\deg(\phi')}=w'(F):=\mx\left(W_{\deg(\phi)}(F)\right).
$$
By the minimality of $\deg(\phi)$, we necessarily have $w'(F)<w(F)$.

An iteration of this argument leads to a finite sequence $$\phi_0,\phi_1,\dots,\phi_r,\phi_{r+1}=F$$ of monic polynomials in $\kx$ such that 
\begin{equation}\label{frameDeg}
1=\deg(\phi_0)<\deg(\phi_1)<\cdots<\deg(\phi_r)<\deg(F),
\end{equation}
whose weighted values $w_i(F):=v(\phi_i(\t))/\deg(\phi_i)$ satisfy: 
\begin{equation}\label{frame}
\deg(g)<\deg(\phi_{i+1})\ \ \Longrightarrow\ \ \dfrac{v(g(\t))}{\deg(g)}\le w_i(F)<w_{i+1}(F),\qquad 0\le i\le r,
\end{equation}
for all monic polynomials $g\in\kx$ of positive degree. 

Note that $w_{r}(F)=w(F)$ and $w_{r+1}(F)=\infty$.

\begin{definition}
An \emph{Okutsu frame} of an Okutsu polynomial $F$, is a list $$[\phi_0,\phi_1,\dots,\phi_r]$$ of monic polynomials in $\kx$ satisfying (\ref{frameDeg}) and (\ref{frame}).

The length $r$ of the frame is called the \emph{Okutsu depth} of $F$. Clearly, the Okutsu depth, the degrees $\deg(\phi_1),\dots,\deg(\phi_r)$, and the values $w_1(F),\dots,w_r(F)=w(F)\in\qg$ are intrinsic data of $F$.
\end{definition}

%\begin{lemma}\label{frameirr}Let $[\phi_0,\phi_1,\dots,\phi_r]$ be an Okutsu frame of an Okutsu polynomial $F$. Then,  $\phi_0,\dots,\phi_r$ are prime polynomials too. \end{lemma}

%\begin{proof}For some $0\le i\le r$, suppose $\phi_i=ab$ for some monic $a,b\in\kx$ with $\deg(a),\deg(b)<\deg(\phi_i)$. By (\ref{frame}), $v(a(\t))/\deg(a),\ v(b(\t))/\deg(b)<C_i(F)$, so that$$C_i(F)=\dfrac{v(\phi_i(\t))}{\deg(\phi_i)}=\dfrac{v(a(\t))+v(b(\t))}{\deg(\phi_i)}<\dfrac{\deg(a)C_i(F)+\deg(b)C_i(F)}{\deg(\phi_i)}=C_i(F),$$which is a contradiction.\end{proof}

For instance, any key polynomial $\phi$ (of degree greater than one) for an inductive valuation $\mu$ is an Okutsu polynomial. Also, any optimal MacLane chain of $\mu$ determines an  Okutsu frame of $\phi$.

\begin{lemma}\label{lower}
Consider an optimal MacLane chain of an inductive valuation:
\begin{equation}\label{depth5}
\minf\stackrel{\phi_0,\ga_0}\lra\  \mu_0\ \stackrel{\phi_1,\ga_1}\lra\  \mu_1\ \stackrel{\phi_2,\ga_2}\lra\ \cdots
\ \stackrel{\phi_{r-1},\ga_{r-1}}\lra\ \mu_{r-1} 
\ \stackrel{\phi_{r},\ga_{r}}\lra\ \mu_{r}=\mu.
\end{equation}
 
Let $F\in\P$ and $\phi\in\kpm$ such that $\phi\mmu F$. Then,  $$v_F(\phi_i)=\ga_i,\quad 0\le i<r,\quad \mbox{ \ and \ }\quad \phi_i\mid_{\mu_{i-1}}F,\quad 0< i\le r.$$
Moreover, if $\phi\nmid_\mu\phi_r$, then $v_F(\phi_r)=\ga_r$ as well. 
\end{lemma}

\begin{proof}
By Theorem \ref{computation}, $\mu<v_F$ and $\ga_i=\mu(\phi_i)=v_F(\phi_i)$ for all $0\le i<r$. In fact, $\phi\nmid_\mu \phi_i$, because $\deg(\phi_i)<\deg(\phi_r)\le\deg(\phi)$.
Also, $\ga_r=\mu(\phi_r)=v_F(\phi_r)$, if  $\phi\nmid_\mu \phi_r$.

By Theorem \ref{fundamental}, $\phi_i\mid_{\mu_{i-1}}F$ for  $0< i\le r$, since $\mu_{i-1}(\phi_i)<\ga_i=\mu(\phi_i)\le v_F(\phi_i)$.
\end{proof}

\begin{theorem}\label{MLOk}
Let $\mu$ be an inductive valuation admitting an optimal MacLane chain as in (\ref{depth5}).
Then, all $\phi\in\kpm$ with $\deg(\phi)>1$ are Okutsu polynomials, and
\begin{enumerate}
\item[(1)] If $\deg(\phi)>\deg(\phi_r)$, then $[\phi_0,\dots,\phi_r]$ is an Okutsu frame of $\phi$. 
\item[(2)] If $\deg(\phi)=\deg(\phi_r)$, then $[\phi_0,\dots,\phi_{r-1}]$ is an Okutsu frame of $\phi$.
\end{enumerate}
Moreover, $w(\phi)=w(\mu)$ in the first case, and $w(\phi)=w(\mu_{r-1})$ in the second case.  
\end{theorem}

\begin{proof}
Suppose $\deg(\phi)>\deg(\phi_r)$. Then, for all $\al\in Z(\phi)$, equation (\ref{mu<vphi}) shows that  
$$
v(\phi_r(\al))=\mu(\phi_r),\qquad v(g(\al))=\mu(g),
$$ 
for all monic $g\in\kx$ with $\deg(g)<\deg(\phi)$. Also, by Theorem \ref{bound}, 
$$
v(g(\al))/\deg(g)=\mu(g)/\deg(g)\le w(\mu)=\mu(\phi_r)/\deg(\phi_r)=v(\phi_r(\al))/\deg(\phi_r),
$$
and equality holds if and only if $g$ is $\mu$-minimal.

By Proposition \ref{extension}, $\phi_r$ is a key polynomial for $\mu$ of minimal degree. By \cite[Thm. 3.7]{KeyPol}, there are no $\mu$-minimal monic polynomials $g$ with  $\deg(g)<\deg(\phi_r)$.  

Therefore, $\phi_r,\phi$ is a distinguished pair, and $w(\phi)=w(\mu)$.

Since the MacLane chain is optimal, we have $\deg(\phi_{i+1})>\deg(\phi_i)$ for all $0\le i<r$, and this argument shows that $\phi_i,\phi_{i+1}$ is a distinguished pair and $w(\phi_{i+1})=w(\mu_i)$.

On the other hand, take $\al_{i+1}\in Z(\phi_{i+1})$. By Lemma \ref{lower}, the tautology $\phi\mmu \phi$ implies $\phi_{i+1}\mid_{\mu_i}\phi$, and Proposition \ref{samevalue} shows that 
$$
g\in\kx,\ \deg(g)<\deg(\phi_{i+1})\ \imp\ v(g(\al_{i+1}))=v(g(\al)).
$$
Thus, $W_{\deg(\phi_{i+1})}(\phi)$ contains a maximal value and $w_{i+1}(\phi)=w(\phi_{i+1})=w(\mu_i)$. 

This ends the proof of (1).

Suppose $\deg(\phi)=\deg(\phi_r)$. By Lemma \ref{lower}, $\phi_r\mid_{\mu_{r-1}}\phi$, so that  $\phi$ is a key polynomial for $\mu_{r-1}$ by Lemma \ref{mid=sim}.
Since $\deg(\phi)>\deg(\phi_{r-1})$, item (2) follows from the previous argument applied to the optimal MacLane chain of $\mu_{r-1}$ deduced by truncation.
\end{proof}

Conversely, any Okutsu frame of an Okutsu polynomial arises in this way.

\begin{theorem}\label{OkML}
Let $F$ be an Okutsu polynomial, and let $[\phi_0,\dots,\phi_r]$ be an Okutsu frame of  $\phi_{r+1}:=F$. 
For all $0\le i\le r$, the mapping
$$
\mu_i\colon \kx\lra\qg\cup\{\infty\}, \qquad \sum\nolimits_{0\le s}a_s\phi_i^s\ \longmapsto \ \mn\{v_F(a_s\phi_i^s)\mid 0\le s\}
$$
 is a valuation admitting $\phi_{i+1}$ as a key polynomial. 
 
Also, $\mu_r$   admits an optimal MacLane chain  as in (\ref{depth5}), with $\ga_i=v_F(\phi_i)$ for all $i$.
%$$\minf\stackrel{\phi_0,v_F(\phi_0)}\lra\  \mu_0\ \stackrel{\phi_1,v_F(\phi_1)}\lra\  \mu_1\ \stackrel{\phi_2,v_F(\phi_2)}\lra\ \cdots\ \stackrel{\phi_{r-1},v_F(\phi_{r-1})}\lra\ \mu_{r-1}\ \stackrel{\phi_{r},v_F(\phi_r)}\lra\ \mu_{r}.$$
\end{theorem}

\begin{proof}
The coefficients $a_s\in K$ of any $\phi_0$-expansion satisfy $v_F(a_s)=v(a_s)$. Hence, $\mu_0=\mu_0(\phi_0,\ga_0)$ by the definition of the depth-zero valuations in section \ref{subsecDepth0}. 

Also, we saw that $\mu_0(\phi_0,\ga_0)$ is equivalent to the augmentation $[\minf;\,\phi_0,(0,\ga_0)]$; thus, $\phi_0$ is a key polynomial for $\mu_0$, by Proposition \ref{extension}. 

Now, suppose that for some $0\le i\le r$, we know that $\mu_i$ is a valuation  and $\phi_i$ is a key polynomial for $\mu_i$ such that $\mu_i(\phi_i)=\ga_i$. 

The theorem will follow from a recursive argument if we deduce that $\phi_{i+1}$ is a key polynomial for $\mu_i$ too; and moreover $\mu_{i+1}=[\mu_i;\phi_{i+1},\ga_{i+1}]$ for $i<r$. 

In fact, $\mu_i<v_F$ by the definition of $\mu_i$. Let $\phi\in\kx$ be a monic polynomial of minimal degree such that $\mu_i(\phi)<v_F(\phi)$. By Lemma \ref{minDegree}, $\phi$ is a key polynomial for $\mu_i$.
Theorem \ref{bound} shows that 
\begin{equation}\label{contrad}
\dfrac{v(\phi(\t))}{\deg(\phi)}>\dfrac{\mu_i(\phi)}{\deg(\phi)}=w(\mu_i)=\dfrac{\mu_i(\phi_i)}{\deg(\phi_i)}=
\dfrac{v(\phi_i(\t))}{\deg(\phi_i)}=w_i(F),
\end{equation}
\begin{equation}\label{contrad2}
\dfrac{v(\phi_{i+1}(\t))}{\deg(\phi_{i+1})}=w_{i+1}(F)>w_i(F)=w(\mu_i)\ge
\dfrac{\mu_i(\phi_{i+1})}{\deg(\phi_{i+1})}.
\end{equation}

By (\ref{frame}) and (\ref{contrad}), we have $\deg(\phi)\ge \deg(\phi_{i+1})$. Also, (\ref{contrad2}) implies
$\deg(\phi)\le \deg(\phi_{i+1})$, by the minimality of $\deg(\phi)$.
Hence, $\deg(\phi)=\deg(\phi_{i+1})$.

By Lemma \ref{minDegree}, $\phi_{i+1}$ is a key polynomial for $\mu_i$ (and it is $\mu_i$-equivalent to $\phi$).

Finally, if $i<r$ then $\mu_{i+1}=[\mu_i;\,\phi_{i+1},\ga_{i+1}]$ by the very definition of the augmented valuation.
\end{proof}

%Finally, let us confirm that ``MacLane chains of inductive valuations" and ``Okutsu frames" are equivalent objects, only attachable to defectless polynomials. 

\begin{theorem}\label{Ok=D1}
For any prime polynomial $F\in\P$ with $\deg(F)>1$, the following conditions are equivalent:
\begin{enumerate}
\item[(1)] $F$ is the key polynomial of an inductive valuation.
\item[(2)] $F$ is an Okutsu polynomial.
\item[(3)] $F$ is defectless.
\end{enumerate}
\end{theorem}

\begin{proof}
 By Theorems \ref{MLOk} and \ref{OkML}, items (1) and (2) are equivalent. Also, 
 (1) implies (3) by \cite{Vaq2}, or \cite[Cor. 5.14]{CompRP}.

Finally, the implication (3) $\Rightarrow$ (1) follows from the  results of Vaqui\'e in \cite{Vaq,Vaq2}. In fact, suppose that $F$ is defectless. As shown in section \ref{secKeyFactors}, there exists a MacLane chain, of arbitrarily large length, of inductive valuations  in the interval $(\minf,v_F)\subset \V$:
$$
\minf\stackrel{\phi_0,\ga_0}\lra\  \mu_0\ \stackrel{\phi_1,\ga_1}\lra\  \cdots
\ \stackrel{\phi_{n-1},\ga_{n-1}}\lra\ \mu_{n-1} 
\ \stackrel{\phi_{n},\ga_{n}}\lra\ \mu_{n}
\ \stackrel{\phi_{n+1},\ga_{n+1}}\lra\  \cdots
$$
with key polynomials satisfying $\phi_{n}\mid_{\mu_{n-1}}F$ for all $n$. 

If $\deg(\phi_{n})=\deg(F)$ for some $n$, then Lemma \ref{mid=sim} shows that $F$ is a key polynomial for $\mu_{n-1}$ and we are done.

Otherwise, since $\deg(\phi_n)\mid \deg(\phi_{n+1})$, this degree becomes constant for a sufficiently large $n$. Thus,  we get a \emph{continuous MacLane chain} which may be augmented to a certain \emph{limit augmented} valuation by using a certain \emph{limit key polynomial} \cite{Vaq}.

This limit augmented valuation lies still in the interval $(\minf,v_F)$, but it is no more inductive. The main result of \cite{Vaq2} shows that in this case $F$ has some defect.
\end{proof}

\noindent{\bf Remark. }The implication (3) $\Rightarrow$  (1) may be deduced from the work by Aghigh and Khanduja too, who use the technique of \emph{complete distinguished chains} of algebraic elements over $K$  \cite{CHF}. \e

The next result follows immediately from Theorems \ref{MLOk}, \ref{OkML} and \ref{Ok=D1}.

\begin{corollary}\label{truncFrame}
Let $F$ be a defectless polynomial with $\deg(F)>1$. 

The sequence $[\phi_0,\phi_1,\dots,\phi_r]$ is an Okutsu frame of $\phi_{r+1}=F$ if and only if $\phi_i,\phi_{i+1}$ is a distinguished pair for all $0\le i\le r$. 

In this case, each $\phi_i$ is a defectless polynomial and $[\phi_0,\dots,\phi_{i-1}]$ is an Okutsu frame of $\phi_i$. 
Moreover, $w_i(F)=w(\phi_{i+1})$ for all $1\le i\le r$. 
\end{corollary}

\subsection{Canonical inductive valuation attached to a defectless polynomial}
Let $F$ be a defectless (prime) polynomial of degree greater than one.

With the notation of Theorem \ref{OkML}, $F$ is a strong key polynomial for the valuation $\mu_r$, which is defined in terms of $\phi_r$-expansions and satisfies $$\mu_r(\phi_r)=v_F(\phi_r)=\deg(\phi_r)w(F),$$
since  $\phi_r,F$ is a distinguished pair.

Since $F$ is $\mu_r$-minimal, Lemma \ref{minimal0} shows that the valuation $\mu_r$ may be defined in terms of $F$-expansions. By (\ref{mu<vphi}) and Theorem \ref{bound}, this valuation is determined by
$$
\as{1.3}
\begin{array}{l}
\mu_r(a)=v_F(a),\quad\forall\,a\in\kx\mbox{ with }\deg(a)<\deg(F),\\ \mu_r(F)=\deg(F)\mu_r(\phi_r)/\deg(\phi_r)=\deg(F)w(F). 
\end{array}
$$

Hence, we may define this valuation $\mu_r$ avoiding any mention to $\phi_r$.

\begin{definition}\label{muF}
Let $F$ be a defectless polynomial with $\deg(F)>1$.

The \emph{Okutsu bound} of $F$ is defined as $\ok(F)=\deg(F)\, w(F)\in\qg$.

There is a canonical inductive valuation $\mu_F$ admitting $F$ as a strong key polynomial, 
determined by the following action on $F$-expansions:
$$
f=\sum\nolimits_{0\le s}a_sF^s\ \imp\ \mu_F(f)=\mn\left\{v_F(a_s)+s\,\ok(F)\mid 0 \le s\right\}. 
$$
\end{definition}

%This valuation $\mu_F$ is a kind of limit of the process of ``approaching $F$ with key polynomials", described in section \ref{secPrimePols}. 

%Let us mention some basic properties of $\mu_F$.
The polynomial $F$ is a key polynomial for infinitely many inductive valuations. Among them, the valuation $\mu_F$ is distinguished by the fact of being minimal.

\begin{lemma}\label{muFval}
Let $F\in\P$ be a defectless polynomial with $\deg(F)>1$. 
\begin{enumerate}
\item The Okutsu depth of $F$ is equal to the MacLane depth of $\mu_F$.
\item If $\phi,F$ is a distinguished pair, then $\phi$ is a key polynomial for $\mu_F$ of minimal degree.
\item The interval $(\mu_F,v_F)\subset\V$ consists of all augmentations
$$
\mu=[\mu_F;\,F,\ga],\qquad \ga\in (\ok(F),\infty)\subset \qg.
$$
\item The valuation $\mu_F$ is the minimal element in the interval  $(\minf,v_F)$ which admits $F$ as a key polynomial.
\end{enumerate}
\end{lemma}

\begin{proof}
Items (1) and (2) are an immediate consequence of Theorem \ref{OkML}.%\e

Let $\mu\in\V$ be any valuation such that $\mu_F<\mu<v_F$. For any $a\in\kx$ with $\deg(a)<\deg(F)$ we have $\mu_F(a)=\mu(a)=v_F(a)$. 

Hence, $F$ is a monic polynomial of minimal degree satisfying $\mu(F)<v_F(F)=\infty$.
By Lemma \ref{minDegree}, $F$ is a key polynomial for $\mu$. By Lemma \ref{shareKP}, 
$\mu=[\mu_F;\,F,\ga]$ for $\ga=\mu(F)>\mu_F(F)=\ok(F)$. This proves item (3).%\e

If a valuation $\mu<\mu_F$ admits $F$ as a key polynomial, then
 $\mu_F=[\mu;\, F,\ok(F)]$ by Lemma \ref{shareKP}. This contradicts item (2), because $F$ would be a key polynomial for $\mu_F$ of minimal degree. 
 Since
the interval  $(\minf,v_F)$ is totally ordered (Theorem \ref{totord}), this argument ends the proof of item (4).
 \end{proof}

\begin{corollary}\label{allinductive}
Let $F\in\P$ be a defectless polynomial with $\deg(F)>1$. Then, all valuations in 
 $(\minf,v_F)$ are inductive.
\end{corollary}

\begin{proof}
Take $\mu\in(\minf,v_F)\subset\V$. Since $\mu_F$ is inductive and $(\minf,\mu_F)$ is totally ordered, we need only to discuss the cases $\mu<\mu_F$ and $\mu>\mu_F$.

If $\mu>\mu_F$, then $\mu$ is inductive because it is an augmentation of $\mu_F$ (Lemma \ref{muFval}). 

If $\mu<\mu_F$, consider a MacLane chain of $\mu_F=\mu_r$ as in Theorem \ref{OkML}. 

If we agree that $\mu_{-1}=\minf$, our valuation $\mu$ fits into
$$
\mu_{i-1}<\mu\le \mu_i=[\mu_{i-1};\phi_i,\ga_i],\quad 0\le i< r.
$$
Since $\mu\le\mu_i<v_{\phi_i}$, Lemma \ref{minDegree} shows that $\phi_i\in \kpm$, because it is a monic polynomial of minimal degree satisfying   $\mu(\phi_i)<v_{\phi_i}(\phi_i)=\infty$. 

Hence, $\mu$ is an augmentation of $\mu_{i-1}$ by Lemma \ref{shareKP}.
\end{proof}

\subsection{Strong types parameterize defectless prime polynomials}\label{subsecTypesDless}
Let us denote by $\Pless$ the set of all defectless prime polynomials in $\kx$ of degree greater than 1.

\begin{lemma}\label{muFmu}
Let $\mu$ be an inductive valuation and $F\in\Pless$. 

Then, $\mu=\mu_F$ if and only if $F$ is a strong key polynomial for $\mu$.
\end{lemma}

\begin{proof}
If $\mu=\mu_F$, then $F$ is obviously a strong key polynomial for $\mu$.

Conversely, if $F\in\kpm$ is strong, then $F$ is a common representative of the two strong types $(\mu,\rrm(F))$, $(\mu_F,\rr_{\mu_F}(F))$. By Lemma \ref{strong=good}, $\mu=\mu_F$.
\end{proof}

\begin{lemma}\label{criteria}
Let $F,G\in\Pless$ be two defectless polynomials of the same degree. The following conditions are equivalent:
\begin{enumerate}
\item[(1)] $v_F(G)>\ok(F)$.
\item[(2)] $F\sim_{\mu_F}G$.
\item[(3)] $\mu_F=\mu_G$ and $\rr(F)=\rr(G)$, where $\rr=\rr_{\mu_F}=\rr_{\mu_G}$.
%\item[(4)] $F$ and $G$ have the same Okutsu frames.  FALS!
\end{enumerate}
If they hold, we say that $F$ and $G$ are \emph{Okutsu equivalent} and we write $F\approx G$.
\end{lemma}

\begin{proof}
Since $\deg(F-G)<\deg(F)$, by the definition of $\mu_F$ we have 
$$\mu_F(F-G)=v_F(F-G)=v_F(G).$$ Since $\ok(F)=\mu_F(F)$, items (1) and (2) are equivalent.

Suppose $F\sim_{\mu_F}G$. Lemma \ref{mid=sim} shows that $G$ is a strong key polynomial for $\mu_F$. By Lemma \ref{muFmu}, $\mu_F=\mu_G$ and Corollary \ref{samefiber} shows that $\rr(F)=\rr(G)$. 

Hence, (2) implies (3).
The opposite implication follows from Corollary \ref{samefiber}.
\end{proof}

The symmetry of condition (3) shows that $\approx$ is an equivalence relation on the set $\Pless$. 
Let us parameterize the quotient set $\Pless/\!\approx$ by an adequate space. 
\newpage

\begin{theorem}\label{MLspace}
Consider the MacLane space $\M=\left\{(\mu,\ll)\in\Tstr\mid \mu\mbox{ inductive}\right\}$, of strong types based on inductive valuations.
The following mapping is bijective:
$$\M\lra \Pless/\!\approx,\qquad \ty=(\mu,\ll)\longmapsto \rep(\ty). 
$$The inverse map is determined by $F\mapsto \left(\mu_F,\rr_{\mu_F}(F)\right)$.
\end{theorem}

\begin{proof}
For any $\ty=(\mu,\ll)\in\M$, the set $\rep(\ty)$ is a $\mu$-equivalence class of strong key polynomials for $\mu$. By Lemma \ref{muFmu}, 
\begin{equation}\label{muphi} 
\mu=\mu_\phi,\qquad \forall\,\phi\in\rep(\ty). 
\end{equation}
By Lemma \ref{criteria},  $\rep(\ty)$ is an Okutsu equivalence class. Therefore, the mapping $\M\to \Pless/\!\approx\,$ is well defined.
By Theorem \ref{caractypes}, it is injective. 

Obviously, $F$ is a  representative of the type $\left(\mu_F,\rr_{\mu_F}(F)\right)$; thus, our mapping is bijective.
By (\ref{muphi}), the inverse mapping is determined by $F\mapsto (\mu_F,\rr_{\mu_F}(F))$. 
\end{proof}

%\subsection{Okutsu equivalent defectless polynomials are close to each other}
By Theorems \ref{MLOk} and \ref{OkML}, two Okutsu equivalent polynomials $F,G\in\Pless$ have the same Okutsu depth, Okutsu frames and
 numerical invariants attached to any optimal MacLane chain of the common inductive valuation $\mu_F=\mu_G$. In particular, as shown in section \ref{subsecMLChains}, they have the same ramification index and residual degree:
$$
e(F)=e(\mu_F)=e(G),\qquad f(F)=f(\mu_F)\deg(R(F))=f(G),
$$
where $R$ is any choice of a residual polynomial operator over $\mu_F$.

Finally, let us shows that two separable Okutsu equivalent polynomials determine two  extensions of $K$ having isomorphic maximal tame subextensions. 

%Let us recall the concept of tameness in our context.

\begin{definition}
A polynomial $F\in\Pless$ is \emph{tame} if the following conditions hold:
\begin{itemize}
%\item $F$ is defectless.
\item The finite extension $k_F/k$ is separable.
\item The ramification index $e(F)$ is not divisible by $\chr(K)$.
\end{itemize}
\end{definition}

It is easy to check that a tame $F$ is necessarily separable over $K$. 

A separable algebraic field extension $L/K$ is tame if every $\t\in L$ is defectless and has a tame minimal polynomial over $K$. 

Recall the definition of the tame ramification subgroup:
$$
G\tame=\left\{\sigma\in\gal(K^s/K)\mid \sigma(c)\sim_vc,\quad \forall\, c\in K^s\right\},
$$
where $K^s\subset \kb$ is the separable closure of $K$.

Its fixed field $(K^s)^{G\tame}\subset K^s$ is the unique maximal tame extension of $K$ in $\kb$.
Also, for any algebraic extension $L/K$, the subfield $$L\tame:=L\cap (K^s)^{G\tame}\subset L$$ is the unique maximal tame subextension of $K$ in $L/K$.

\begin{theorem}\label{tame}
 Let $F,G\in\Pless$ be separable and Okutsu equivalent. For $\t\in Z(F)$, take $\om\in Z(G)$ such that
 $ v(\t-\om)=\mx\left\{v(\t-\om')\mid \om'\in Z(G)\right\}$.
 
 Then, $K(\t)\tame= K(\om)\tame$.
\end{theorem}

\begin{proof}
 By Lemma \ref{criteria}, $G$ is a key polynomial for $\mu_F$ and $G\mid_{\mu_F}F$. 
 By Lemma \ref{averageToSingle}, 
 $$
 \beta\in\kb,\quad \deg_K(\beta)<\deg(G)\ \imp\ v(\t-\beta)<v(\t-\om).
 $$
As we argued along the proof of Proposition \ref{samevalue}, this implies 
\begin{equation}\label{g}
g(\t)\sim_v g(\om),\qquad \forall\,g\in\kx\mbox{ with }\deg(g)<\deg(G).
\end{equation}

Let $M/K$ be a finite Galois extension containing $\t$ and $\om$. 
The subfields $K(\t)$ and $K(\om)\tame\subset M$ are the fixed fields of the following subgroups of $\gal(M/K)$:
$$
H_\t=\{\sigma\in H\mid \sigma(\t)=\t\},\qquad
H\tame_\om=\{\sigma\in H\mid \sigma(c)\sim_v c,\quad \forall\,c\in K(\om)\}.
$$

Any $c\in K(\om)$ may be written as $c=g(\om)$ for some $g\in\kx$ with $\deg(g)<\deg(G)$. Hence, $H_\t\subset H\tame_\om$, because for all $\sigma\in H_\t$, (\ref{g}) shows that
$$
\sigma\left(g(\om)\right)\sim_v \sigma\left(g(\t)\right)=g(\t)\sim_v g(\om).
$$

By Galois theory, $K(\om)\tame\subset K(\t)$, and this implies  $K(\om)\tame\subset K(\t)\tame$, by the maximality of $K(\t)\tame$. 
By the symmetry of the argument,  $K(\t)\tame=K(\om)\tame$.
 \end{proof}

\begin{corollary}
Suppose that $F,G\in\Pless$ are separable and $F\approx G$. If $F$ is tame, then  for all $\t\in Z(F)$ there exists a root $\om\in Z(G)$ such that $K(\t)=K(\om)$.
\end{corollary}

\end{document}